\documentclass[11pt]{article}

\usepackage[psamsfonts]{amssymb}
\usepackage{amsfonts,amsmath}
\usepackage{amsthm}

\usepackage[backend=biber,style=numeric,autocite=inline,maxbibnames=100]{biblatex}
\usepackage{listings}
\usepackage{xcolor}
\usepackage{epsfig,multicol}

\newcommand{\IH}{\mathbb{H}}

\newcommand{\IR}{{\mathbb{R}}}

\newcommand{\IS}{\mathbb{S}}

\newcommand{\IZ}{\mathbb{Z}}

\newcommand{\IC}{\mathbb{C}}
\newcommand{\ID}{\mathbb{D}}
\newcommand{\oC}{\hat{\IC}}

\newcommand{\abs}[1]{\lvert#1\rvert}

\let\Re\relax
\let\Im\relax

\DeclareMathOperator{\PSL}{PSL}
\DeclareMathOperator{\Isom}{Isom}
\DeclareMathOperator{\SL}{SL}
\DeclareMathOperator{\tr}{tr}
\DeclareMathOperator{\Re}{Re}
\DeclareMathOperator{\Im}{Im}

\newtheorem{theorem}{Theorem}
\newtheorem*{theorem*}{Theorem}
\newtheorem{lemma}{Lemma}

\newtheorem{corollary}{Corollary}
\newtheorem{conjecture}{Conjecture}

\theoremstyle{definition}

\theoremstyle{remark}
\newtheorem*{remark}{Remark}

\newtheorem*{notation}{Notation}

\usepackage{hyperref}

\title{Approximations of the Riley slice}
\author{Alex Elzenaar, Gaven Martin \& Jeroen Schillewaert \thanks{
Parts of this work appear in the MSc thesis of the first author.
\newline
Work of authors partially supported by the New Zealand Marsden Fund and by a University of Auckland FRDF grant.
 \newline
AE+JS, Department of Mathematics,  The University of Auckland,
\newline
email: \url{aelz176@aucklanduni.ac.nz},\url{j.schillewaert@auckland.ac.nz}
 \newline
GM, Institute for Advanced Study,
Massey University,  Auckland,
New Zealand.
\newline
email: \url{g.j.martin@massey.ac.nz}
\newline
{\bf MSC Subject.} 32G15, 30F40 57N05
}}
\date{}
\begin{document}
\maketitle
\begin{abstract}
Adapting the ideas of L. Keen and C. Series used in their study of the Riley slice of Schottky groups generated by two parabolics,  we explicitly identify `half-space' neighbourhoods of pleating rays which lie completely in the Riley slice.  This gives a provable method to determine if a point is in the Riley slice or not.  We also discuss the family of Farey polynomials which determine the rational pleating rays and their root set which determines the Riley slice; this leads to a dynamical systems interpretation of the slice.

Adapting these methods to the case of Schottky groups generated by two elliptic elements in subsequent work facilitates the programme to identify all the finitely many arithmetic generalised triangle groups and their kin.
\end{abstract}


\section{Introduction}
The usual route to the Riley slice is through the theory of Kleinian groups,  or more precisely Schottky groups,  generated by two parabolic elements in $\PSL(2,\IC)$.  This problem has a long
history begining with Sanov \cite{Sanov} in 1947, Brenner \cite{Brenner} and Chang, Jennings, and Ree \cite{CJR} and Lyndon and Ullman \cite{LU}.  These papers all relied on various versions of
what are now known as combination theorems and these are explained in some detail in Maskit's book \cite{Mas}; this book also includes the basic theory of Kleinian groups, which we briefly discuss
now in order to fix notation.

Recall that a Kleinian group may be equivalently defined as (a) a discrete subgroup of $ \PSL(2,\mathbb{C}) $, or (b) a discrete subgroup of $ \Isom^+(\mathbb{H}^3) $. The relationship between these
two definitions comes from the fact that isometries of hyperbolic 3-space are uniquely characterised by their actions on the sphere at infinity: namely, there is a natural bijection between $ \Isom^+(\mathbb{H}^3) $
and the group of conformal automorphisms of $ S^2 $. After identifying $ S^2 $ with the Riemann sphere $ \hat\IC $, we may characterise the conformal automorphisms as none other than the M\"obius
transformations, the maps of the form $ z \mapsto \frac{az + b}{cz + d} $ ($a,b,c,d \in \mathbb{C} $ with $ ad-bc \neq 0 $). Performing one final identification, of $ \hat\IC $ with $ \mathbb{P}\mathbb{C}^1 $,
we see that the M\"obius transformations are in natural correspondence with $ \SL(2,\mathbb{C}) $ via the identification
\begin{displaymath}
  \left( z \mapsto \frac{az + b}{cz + d} \right) \leftrightarrow \begin{pmatrix} a&b\\c&d \end{pmatrix}.
\end{displaymath}
Observe finally that $ \SL(2,\mathbb{C}) = \{\pm 1\} \times \PSL(2,\mathbb{C}) $, but in the world of M\"obius transformations we can always multiply both the numerator and denominator through by $ \sqrt{-1} $
if necessary to normalise the denominator of any matrix representative to 1 without changing the geometry of the map --- in other words, we can always lift from $ \PSL(2,\mathbb{C}) $ to $ \SL(2,\mathbb{C}) $
without issue as long as we are careful to always pick representatives of determinant 1.

The Riley slice $ \mathcal{R} $ is the moduli space parameterising the complex structures on the four-times punctured sphere $ S^2_4 $. More precisely, define a
family $ (\Gamma_\mu)_{\mu \in \mathbb{C} \setminus \{0\}} $ of subgroups of $ \PSL(2,\IC) $ by
\begin{displaymath}
  \Gamma_\mu := \left\langle X = \begin{pmatrix} 1 & 1 \\ 0 & 1 \end{pmatrix}, Y_\mu = \begin{pmatrix} 1 & 0 \\ \mu & 1 \end{pmatrix} \right\rangle;
\end{displaymath}
The assumption $\mu\neq 0$ implies that $\Gamma_\mu$ is not abelian and has free subgroups of all ranks. The group $\Gamma_\mu$ acts on the Riemann sphere $\oC$ and there is a largest open (possibly empty)
set $ \Omega(\Gamma_\mu) \subset \IC$ on which this group acts discontinuously (the \textit{ordinary set}); the complement of this set in $ \hat\IC $ is the \textit{limit set} $ \Lambda(\Gamma_\mu) $
and is the closure of the set of fixed points of elements of $ \Gamma_\mu $ (so, since $ \infty $ is fixed by $ f $, $ \infty \in \Lambda(\Gamma_\mu) $).\footnote{One may also define the ordinary set in the following
way, if $ \Gamma_\mu $ is discrete and non-elementary (true for every group in $ \mathcal{R} $): it is the largest domain in $ \mathcal{C} $ on which the transformations of $ \Gamma_\mu $ are equicontinuous. In this
way the ordinary set is analogous to the \textit{Fatou set} of a dynamical system, as in Section~\ref{sec:dynamics} below.}
The quotient $\Omega(\Gamma_\mu)/\Gamma_\mu$ is a Riemann surface. When $ \Gamma_\mu $ is free and discrete, the Riemann surfaces so obtained are supported on one of three homeomorphism
classes of topological space: the empty set; a disjoint union of two three-times punctured spheres; and a four-times punctured sphere \cite{MS}. It happens that the first two types of
space may be viewed as geometric deformations of four-times punctured spheres, and so it is natural to consider the set of all $ \mu $ such that $ \Gamma_\mu $ is free and discrete and
such that this quotient is a four-times punctured sphere; the other two kinds of space then form the boundary of this set.

Thus, the Riley slice is defined by
\begin{displaymath}
  \mathcal{R} = \{\mu\in \IC : \Omega(\Gamma_\mu)/\Gamma_\mu \text{ is topologically a four-times punctured sphere} \}.
\end{displaymath}

Denote the M\"obius transformations of $\oC$ representing $X$ and $Y$ by $f$ and $g$ respectively, so
\begin{displaymath}
  f(z) = z + 1 \text{ and } g(z) = \frac{z}{\mu z + 1}
\end{displaymath}
for $ z \in \IC $. We will abuse notation and write $ \Gamma_\mu $ for $ \langle f,g \rangle $ as well as for the matrix group.

\begin{notation}
  Often we will be working with a fixed $ \Gamma_\mu $; in this situation we will often write $ \Gamma $ and  $ Y $ for $ \Gamma_\mu $ and $ Y_\mu $ without comment.
\end{notation}

One can also view the Riley slice as the quotient of the Teichm\"uller space ${\cal T}_{0,4}$ (genus $0$ surface with $4$ punctures) by a Dehn twist about a simple
closed curve which separates one pair of punctures from another in $S^2_4$.  Since ${\cal T}_{0,4}$ is simply connected,  and the Dehn twist generates an infinite
cyclic group acting effectively and discretely on ${\cal T}_{0,4}$, it is the mapping class group,  see \cite{Lyb}. Thus the Riley slice is topologically an annulus
and admits an intrinsic hyperbolic metric which agrees with the Teichm\"uller metric.

The theory of Keen and Series \cite{KS} endows the Riley slice with a foliation structure. This structure consists of a set of curves parameterised by $ \mathbb{Q} $ which radiate out from the boundary of the slice
and which are dense in the slice (the so-called \textit{rational pleating rays}) together with a natural completion (in the sense that we may add curves parameterised by $ \mathbb{R}\setminus\mathbb{Q} $ in order to fill out the entire slice). In Figure~\ref{fig:pleating_rays} we illustrate the Riley slice together with a selection of rational pleating rays.

The exterior of the Riley slice (the bounded region of Figure~\ref{fig:pleating_rays}) is also of interest: it includes all the groups $\Gamma_\mu$ which are discrete but not free,  among them are, for instance, all hyperbolic two bridge knot complements. These lie along or at the endpoint of a rational pleating ray.   Recently \cite{A1,A2} gave a complete description of all these discrete groups outside the Riley slice as Heckoid groups and their near relatives.  For each such group there are at most two Nielsen classes of parabolic generating pairs.  The boundary of the Riley slice is a Jordan curve with outward directed cusps \cite{A3}.  The non-discrete groups are generically free,  but every neighbourhood of a non-discrete group contains a supergroup containing any given two parabolic groups --- discrete or otherwise ---  and a group with any prescribed number of distinct Nielsen classes, \cite{M1}.

\begin{figure}
  \centering
  \includegraphics[width=\textwidth]{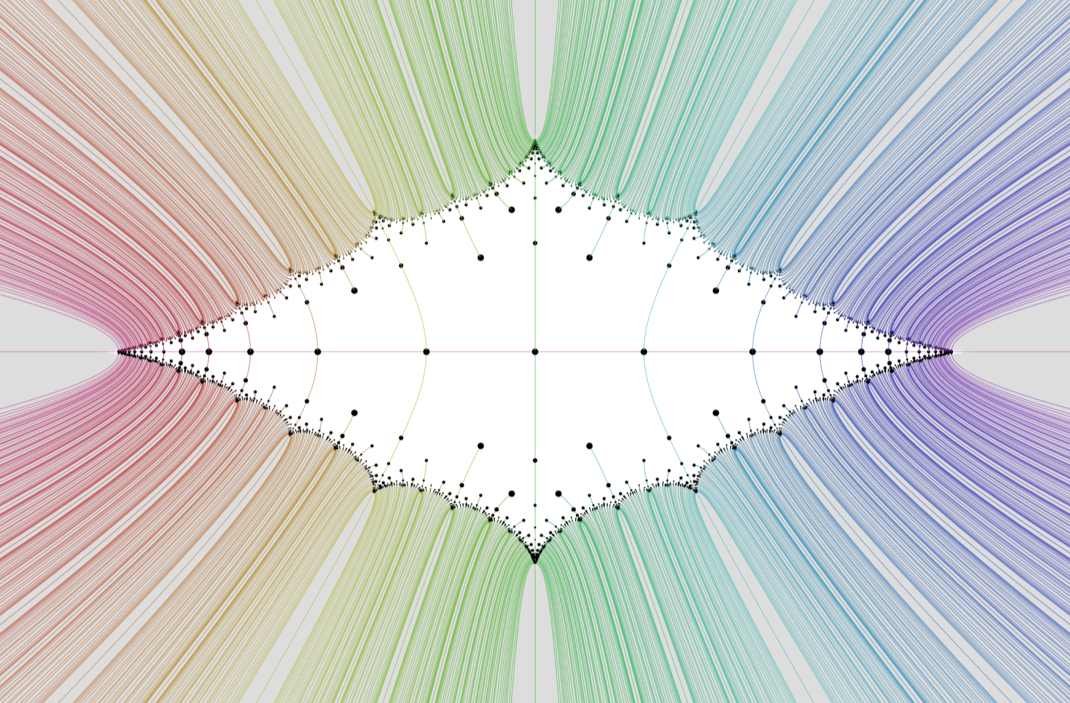}
  \caption{The Riley slice is the unbounded region, foliated by  `rational pleating rays' (the coloured curves).  The symmetries of this space include complex conjugation and $\mu \leftrightarrow -\mu$ interchanging $g$ with $g^{-1}$. Picture courtesy of Y. Yamashita.\label{fig:pleating_rays}}
\end{figure}

Our motivation for the study of the Riley slice here is the continuation of a longstanding programme to identify all the finitely many generalised arithmetic triangle groups in $\PSL(2,\IC)$ \cite{GM,GMM,CMM,MM,MSY}.  For this programme, one needs quite refined computational descriptions of other one complex dimensional moduli spaces such as the moduli space of $S^2_{p,q}$,  the 2-sphere with four cone points (two of order $p$ and two of order $q$).  Arithmetic criteria developed and described in \cite{GMMR} identify those algebraic integers in $\IC$ which give rise to discrete subgroups of arithmetic lattices.  Obtaining degree bounds, and numerically identifying these points, is a challenging task,  see \cite{FR,MM,MSY}. Once an algebraic integer is identified there are further problems.  \textit{A priori}, the relevant group is discrete,  but we need to know if it is in fact free on the generators (analogue of Riley slice) and if it is not,  identify the abstract group and hyperbolic 3-manifold quotient.  We illustrate some relevant data in Figure~\ref{fig:yamashita}.

\begin{figure}
  \centering
  \includegraphics[width=\textwidth]{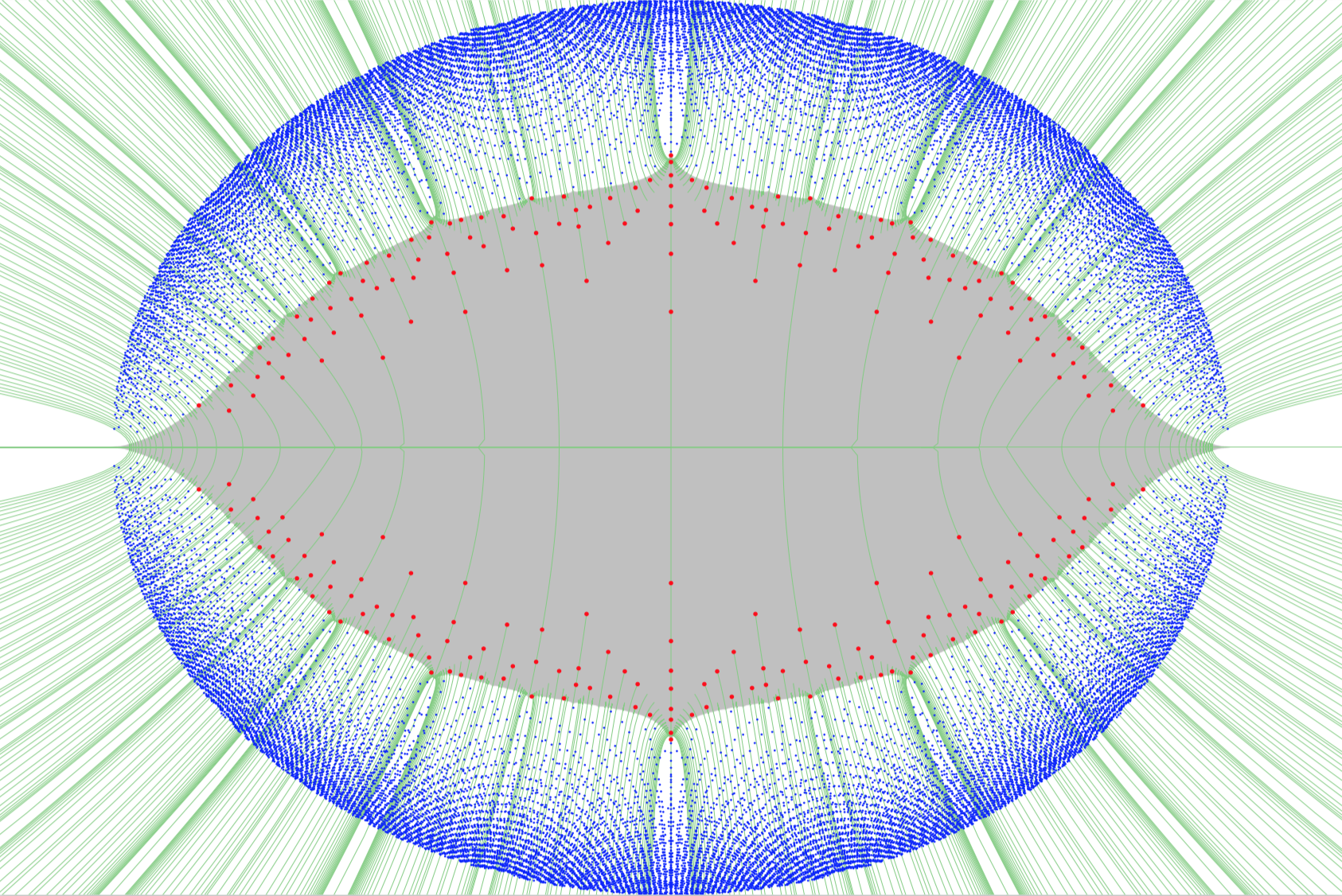}
  \caption{$\mathcal{R}_{3,3}$. The analogue of the Riley slice is the exterior of the grey region and foliated by rational pleating
           rays.  The $15,909$ algebraic integers satisfying the arithmetic criteria described found by Flammang and Rhin \cite{FR}
           above are blue (if, by visual inspection, are outside the grey region) and red (if, by visual inspection, are inside the
           grey region).  These red points yield either lattices or rigid groups with circle packing limit sets in $\PSL(2,\IC)$ that
           are generated by elliptic elements of order two and three. Picture in collaboration with Y. Yamashita.\label{fig:yamashita}}
\end{figure}

In order to be able to resolve these issues we need to be able to \emph{provably} decide if a point in $\IC$ actually lies in the Riley slice or its analogue.
Solving this problem also has computational applications,  including identifying all the finitely many generalised arithmetic triangle groups in $\PSL(2,\IC)$ \cite{GM,GMM,CMM,MM}.

\paragraph{Main results.}
Our first main result, Theorem~\ref{thm:rec} (p.\pageref{thm:rec}), sets up a dynamical system whose stable region contains $ \mathcal{R} $; this gives a system of polynomials which we
call $ Q_{p/q} $ whose filled Julia sets lie in the exterior of the Riley slice. As a incidental consequence of the theory used to derive this result, we characterise the Farey word traces
of the discrete groups which lie on pleating ray extensions (Theorem~\ref{thm:discrete_gps_on_pleating_rays}).

With the technology of Keen and Series, we may identify whether a point lies on a rational pleating ray, but the union of these rays has measure $0$ so it is not so useful to check whether a point
lies in the Riley slice. In this paper, we show that a well defined open neighbourhood of each rational pleating ray lies in the Riley slice (or its torsion generated analogue) so that we can `capture' points.
This is the content of our second main theorem, Theorem~\ref{main} (p.\pageref{main}), which extends the theory of Keen and Series to give such neighbourhoods.

\paragraph{Structure of the paper.}
In section~\ref{sec:FWP} we introduce the Farey words, which represent simple closed curves on the four-times punctured sphere which are not boundary-parallel; the basis of the theory is the
relationship between the combinatorics
and algebra of these words and the deformations of the curves that they represent. In Section~\ref{sec:FP} we define the Farey polynomials $P_{p/q}(\mu)$ and give our first main result, Theorem~\ref{thm:rec},
together with some conjectures on the structure of the Farey polynomials and the dynamical system that they generate.
In Section~\ref{sec:NRPR} we motivate our second main result, Theorem~\ref{main}, and place it in context with prior work by Lyndon and Ullman. We conclude the paper in Section~\ref{sec:PF} with
the proof of Theorem~\ref{main} together with some related estimates.

\paragraph{Future work.}
In an upcoming paper \autocite{EMSElliptic} we extend the theory of neighbourhoods of cusp points to the elliptic setting (an example of an elliptic Riley slice is given in Figure~\ref{fig:ellipticslice}); this
is more than just a slight modification of the argument for the parabolic case and it relies on more accurate estimates like those hinted at in Figure~\ref{fig:LURays} on page~\ref{fig:LURays} below.

\begin{figure}
  \centering
  \includegraphics[width=0.6\textwidth]{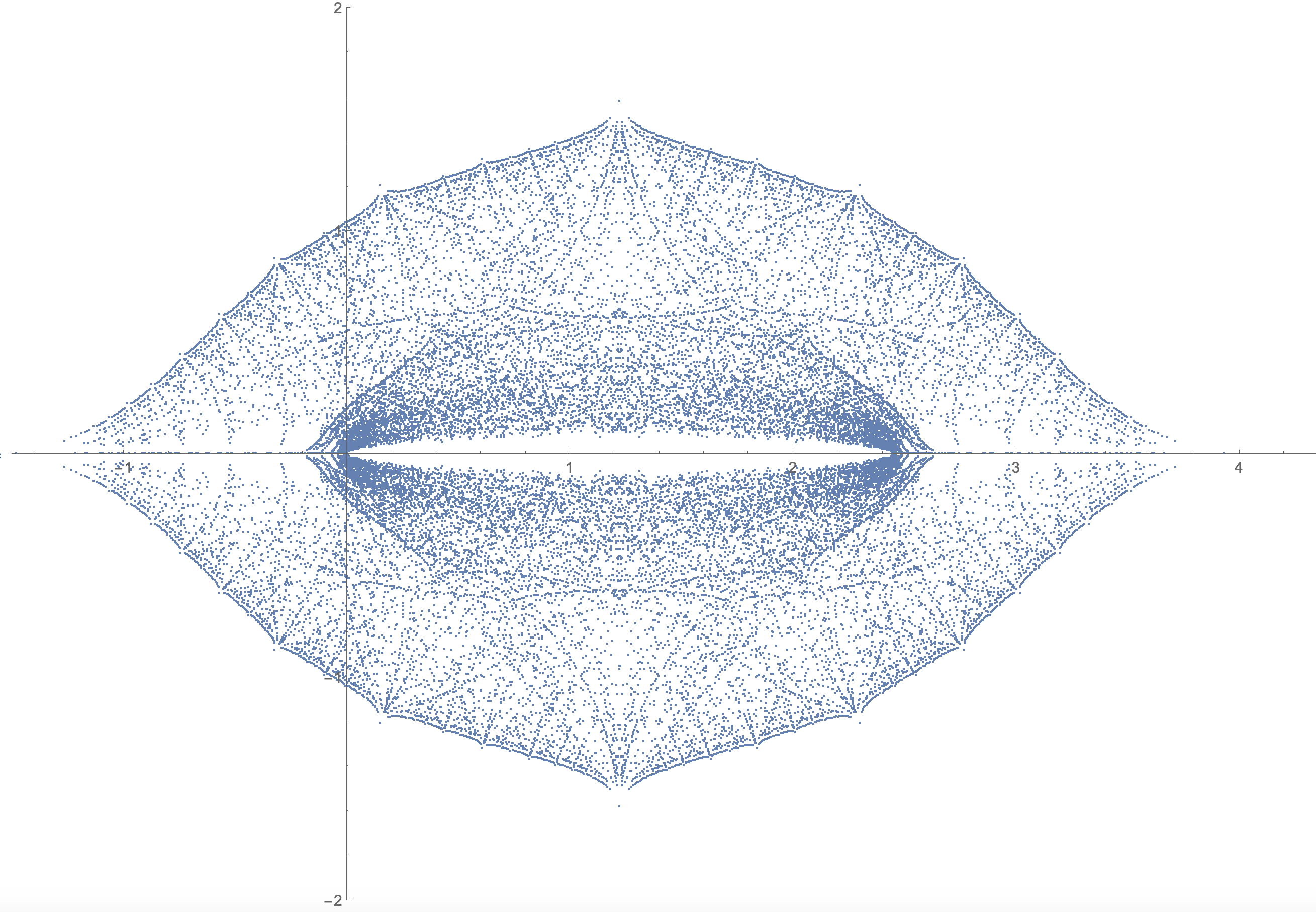}
  \caption{Points approximating the slice of Schottky space corresponding to a sphere with paired cone points, of respective cone angles $ 2\pi/6 $ and $ 2\pi/8 $.\label{fig:ellipticslice}}
\end{figure}

In preparation we also have a paper \autocite{EMSCombinatorics} which gives various combinatorial identities involving the Farey words and their trace polynomials, including a much more efficient
method for determining the Farey polynomials computationally via a recurrence formula, along with some applications to the geometry of the Riley slice boundary. Closed forms for certain sequences
of Farey polynomials can also be computed, allowing the approximation of the Riley slice near the cusp point at $ +4 $ by a sequence of well-behaved neighbourhoods of the form described in
Section~\ref{sec:NRPR} of the current paper. In Figure~\ref{fig:cuspapprox} we include a picture produced from the first 100 polynomials of this approximating sequence; it took just under two minutes to
compute the roots symbolically in \texttt{Mathematica} from scratch.

\begin{figure}
  \centering
  \includegraphics[width=0.6\textwidth]{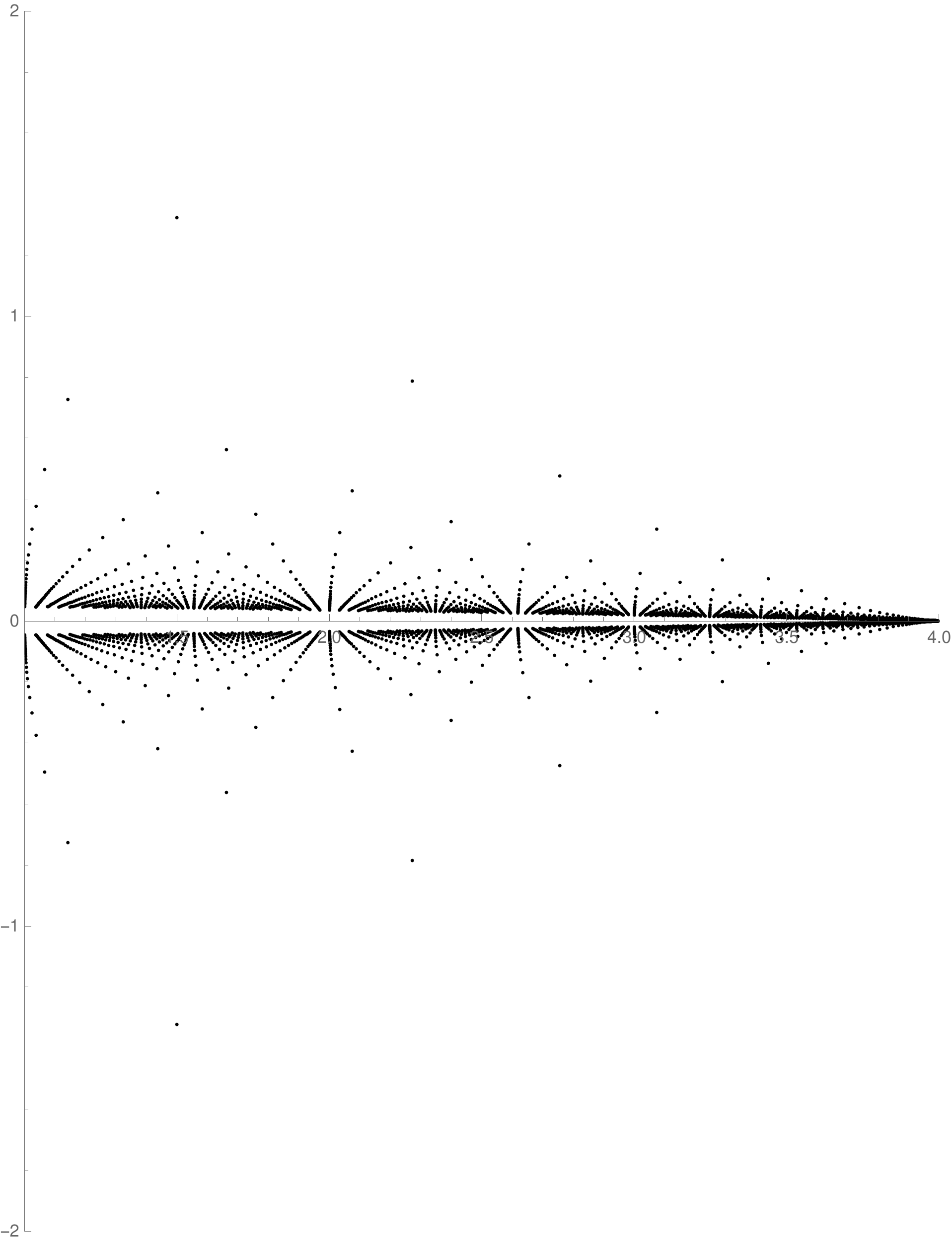}
  \caption{Inverse images of $ -2 $ under a sequence of Farey polynomials approximating the $ +4$-cusp.\label{fig:cuspapprox}}
\end{figure}

\section{Farey words and polynomials}\label{sec:FWP}

\begin{figure}
  \centering
  \includegraphics[width=0.8\textwidth]{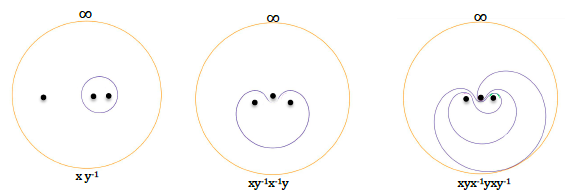}
  \caption{Simple closed curves on a 4-times punctured sphere,  from left: $\frac{1}{1}$, $\frac{1}{2}$, $\frac{2}{3}$.  One puncture is at $\infty$,  disks bound neighbourhoods of the other punctures.\label{fig:SCC}}
\end{figure}

A \textit{Farey word} is a word in $x$ and $y$ representing a simple closed curve on the four-times punctured sphere which is not homotopic to a cusp (Figure~\ref{fig:SCC}).
The definition of these words in terms of rational slopes $p/q$ is explained in \cite[\S 2.3]{KS} with some corrections in \cite{SM}.
The exact details are not useful to us here; however, it will be useful to know the broad structure of the Farey words. The main result is the following, which is
essentially immediate from the combinatorial definition given in \cite{KS}.

\begin{lemma}\label{lemma1}
  Let $p/q$ be a rational slope and $W_{p/q}$ a Farey word. Then $ W_{p/q} $ has word length $ 2q $. Further:
  \begin{enumerate}
    \item If $q$ is even, then there are $u,v\in \langle x,y\rangle$ such that
          \begin{displaymath}
            W_{p/q} = x u x^{-1} u^{-1} = v y v^{-1} y^{-1}.
          \end{displaymath}
    \item If $q$ is odd, then there are $u,v\in \langle x,y\rangle$ such that
          \begin{displaymath}
            W_{p/q} = x u y u^{-1} = v x^{-1} v^{-1} y^{-1}
          \end{displaymath}
  \end{enumerate}
  In particular if $q$ is even,  then $W_{p/q}$ is a commutator in two different ways. The word length of $W_{p/q}$ is $2q$. \qed
\end{lemma}

We can view $ W_{p/q} $ as a word $ W_{p/q}(\mu) $ in $ \Gamma_\mu $ by performing the substitution $ x \mapsto X $, $ y \mapsto Y_\mu $:
\begin{equation}\label{eqn:word_coeffs}
  W_{p/q}(\mu) = \begin{pmatrix} a_{p/q}(\mu) & b_{p/q}(\mu) \\ c_{p/q}(\mu) & d_{p/q}(\mu) \end{pmatrix} \qquad  a_{p/q}d_{p/q}-b_{p/q}c_{p/q}=1.
\end{equation}
The entries of $W_{p/q}(\mu)$ are polynomials of degree $q$ in the symbol $ \mu $. In particular, the trace $ \tr W_{p/q}(\mu) $ is a polynomial
of degree $q$ in $ \mu $; we call this polynomial the \textit{Farey polynomial} of slope $ p/q $ and denote it by $ P_{p/q}(\mu) $. The polynomial
$ Q_{p/q} := P_{p/q} - 2 $ also turns out to be very useful in the sequel. In Table~\ref{tab:farey_words}, we list examples of Farey words with
small-denominator slopes, together with their corresponding polynomials.

\begin{notation}
  Just as we write $ f $ and $ g$ for the M\"obius transformations associated to $ X $ and $ Y $, we write $ h_{p/q}(\mu) $ for the M\"obius
  transformation associated to $ W_{p/q}(\mu) $.
\end{notation}

\begin{table}
  \centering
  \caption{Farey words and their corresponding cusp points.\label{tab:farey_words}}
  \resizebox{\textwidth}{!}{\begin{tabular}{ |c|c|c|c|c|}
    \hline
    $p$ & $q$ & Farey word $W_{p/q}$ & $Q_{p/q}(\mu)  = P_{p/q}(\mu)-2$ & Approx. cusp point \\
    \hline
    $1$ & $2$ & $x  y  x^{-1}  y^{-1}$ & $\mu ^2$ & $2i$\\
    $4$ & $7$ & $x  y  x^{-1}  y^{-1}  x  y  x^{-1}  y  x  y^{-1}  x^{-1}  y  x  y^{-1}$ & $\mu \left(-1+2\mu - \mu^2+\mu^3\right)^2$ & $0.427505\, +1.57557 i$ \\
    $3$ & $5$ & $x  y  x^{-1}  y^{-1}  x  y^{-1}  x^{-1}  y  x  y^{-1}$ & $-\mu(1 -\mu + \mu^2)^2$ & $0.773301\, +1.46771 i$ \\
    $5$ & $8$ & $x  y  x^{-1}  y^{-1}  x  y^{-1}  x^{-1}  y  x^{-1}  y^{-1}  x  y  x^{-1}  y  x  y^{-1}$ & $\mu^4 (2 - 2 \mu+ \mu^2)^2$ & $1.05642\, +1.30324 i$ \\
    $2$ & $3$ & $x  y  x^{-1}  y  x  y^{-1}$ & $z(z-1 )^2$ &$1.5\, +(\sqrt{7}/2) i $\\
    $5$ & $7$ & $x  y  x^{-1}  y  x  y^{-1}  x  y^{-1}  x^{-1}  y  x^{-1}  y^{-1}  x  y^{-1}$ & $-\mu(1 + 2\mu - 3\mu^2 + \mu^3)^2$ & $1.85181\, +0.911292 i$ \\
    $3$ & $4$ & $x  y  x^{-1}  y  x^{-1}  y^{-1}  x  y^{-1}$ & $\mu^2(\mu-2)^2$ & $2.27202\, +0.786151 i$\\
    $4$ & $5$ & $x  y  x^{-1}  y  x^{-1}  y  x  y^{-1}  x  y^{-1}$ & $\mu(1 - 3\mu + \mu^2)^2$ & $2.75577\, +0.474477 i$ \\
    $1$ & $1$ & $x  y^{-1}$ & $-\mu$ &$4$\\
    \hline
  \end{tabular}}
\end{table}

Our computational exploration of the matrices suggested the following result.

\begin{theorem}\label{thmiso}
  With the notation of Equation~(\ref{eqn:word_coeffs}),
  \begin{equation}\label{iso}
    Q_{p/q}(\mu) = a(\mu)+d(\mu)-2=c(\mu).
  \end{equation}
\end{theorem}
\begin{proof}
  Using Lemma~\ref{lemma1} we will show this reduces to the well known Fricke identity in $\PSL(2,\IC)$,
  \begin{displaymath}
  \tr[a,b] = \tr^2 (a) + \tr^2 (b) + \tr^2 (ab) - \tr (a) \tr (b) \tr (ab) -2.
  \end{displaymath}

  We put $a=X^{-1}$ and $b=W_{p/q}$.  Note that, by Lemma~\ref{lemma1} and the conjugacy invariance of trace, $\tr X^{-1}W_{p/q} = \tr(X)$ or $\tr(Y)$ depending on whether $q$ is even or odd.
  In our situation both of these numbers are $2$. Thus, supposing $q$ is odd,
  \begin{align*}
    c_{p/q}^2 & = \tr[X,W_{p/q}]-2   =  \tr[X^{-1},W_{p/q}] -2 \\
              & =   \tr^2 (X^{-1}) + \tr^2 (W_{p/q}) + \tr^2 (Y) - \tr (X^{-1}) \tr (W_{p/q}) \tr (Y) -4\\
              & = 4 + (a_{p/q}+d_{p/q})^2   - 4 \tr (a_{p/q}+d_{p/q}) \\
              & = (a_{p/q}+d_{p/q}    - 2)^2
  \end{align*}
  Thus $c_{p/q}=\pm(a_{p/q}+d_{p/q} - 2)$. When $\mu = 1$, the positive square root occurs. Since the identity is continuous in $\mu$, it follows that the positive square root is the correct choice for all $\mu$.
  The result follows with a similar calculation if $q$ is even.
\end{proof}

\begin{remark}
  In fact as an identity among polynomials  in $\mu$ we only need to check Equation~(\ref{iso}) on the integers.  If $\mu\in \IZ$,  then the group $\langle X,W_{p/q} \rangle$ is a
  subgroup of $\PSL(2,\IZ)$ (or in fact of the modular group) and the statement of the theorem is simply that the isometric circles of $W_{p/q}$ are tangent to those of $X$ (the
  latter being the vertical lines $\Re(z)=-\frac{1}{2}$, and $\Re(z)=+\frac{1}{2}$.  This can also be seen from the Farey tessellation of the interval $[0,1]$ illustrated in
  Figure~\ref{fig:farey_tess},  see \cite{Farey}.
\end{remark}

\begin{figure}
  \centering
  \includegraphics[width=0.5\textwidth]{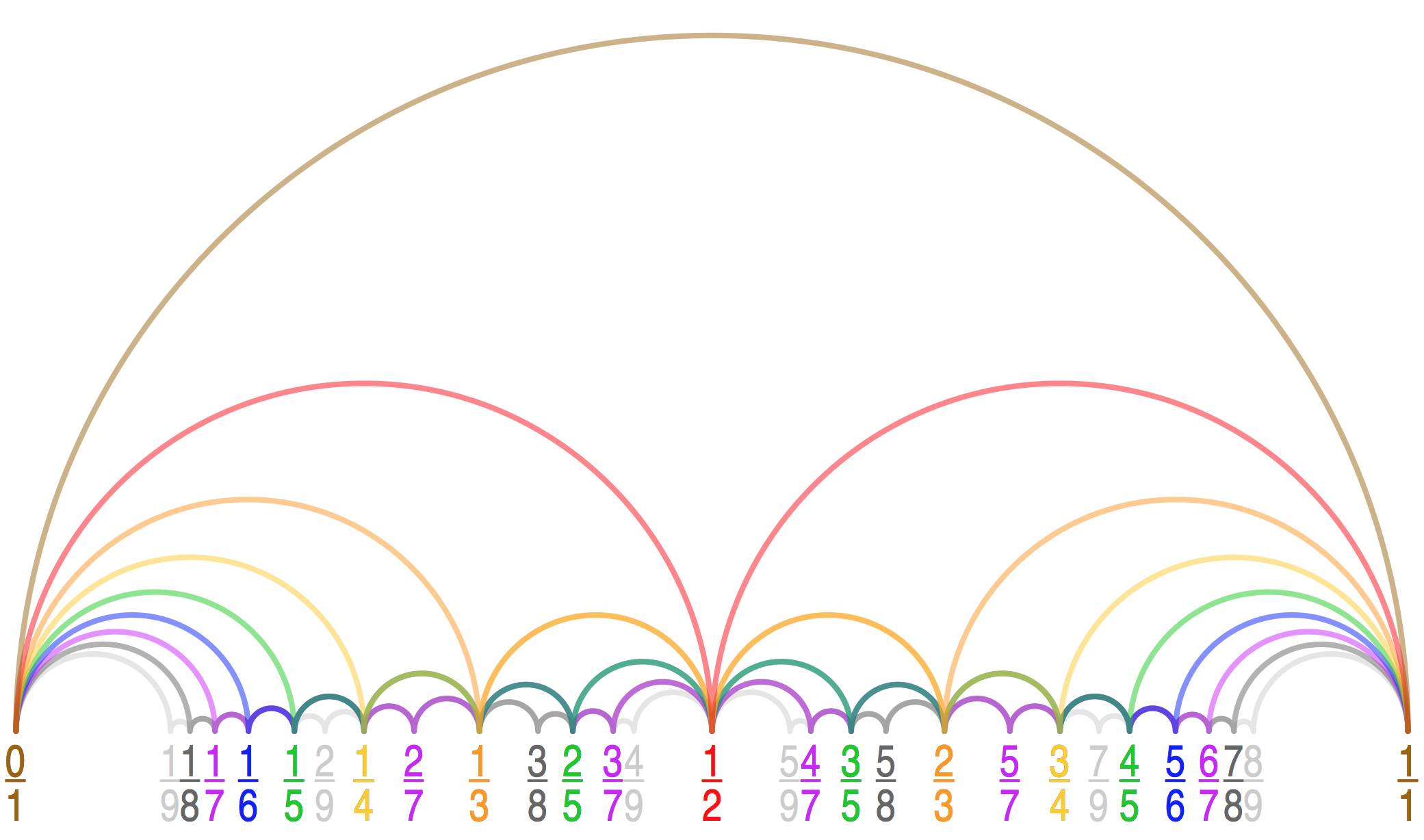}
  \caption{The Farey tessellation of $[0,1]$.\label{fig:farey_tess}}
\end{figure}

The importance of Farey words in this setting is that in order for an isomorphic family of discrete groups to approach the boundary of a moduli space, a simple closed curve has to shrink to a cusp.
That is, a word in the reference group has to become parabolic. The limit of a sequence of finitely generated Kleinian groups (where the number of generators is fixed) with generators converging is
again a Kleinian group by J\o rgensen's algebraic convergence theorem \cite{Jorg}.  Thus we have the following result:
\begin{lemma}\label{lem:bdry_discrete}
  All the points in the Riley slice boundary represent discrete groups. \qed
\end{lemma}

The groups on the boundary of the Riley slice for which $ \Omega(\Gamma_\mu)/\Gamma_\mu$ is a disjoint union of triply punctured spheres (the surface that is naturally obtained by shrinking a
simple closed curve on $S_4^2$)  are called \textit{cusp groups}. A point in $\partial\mathcal{R}$ which is not a cusp group has empty ordinary set (since the quotient cannot support moduli) and is degenerate \cite{Bers}.

Parabolic M\"obius transformations are easily identified by the trace condition,
\begin{displaymath}
  \beta(f)=\tr^2(f)-4 = 0 \mbox{ if and only if $f$ is parabolic.}
\end{displaymath}
Here,  and in what follows,  we have abused notation and written $\tr(f)$ for the trace of the matrix representative of $f$ in $\PSL(2,\IC)$.  Keen and Series \cite{KS} study the boundary of the Riley
slice by considering what happens for a fixed slope $p/q$ as $\tr(f_{p/q})\to -2$,  $\tr(f_{p/q})\in \IR$.   In fact Keen and Series show that the Farey polynomial $ P_{p/q} $
has a branch so that the pleating ray
\begin{displaymath}
  \mathcal{P}_{p/q} = P_{p/q}^{-1}\big((-\infty,-2]\big)
\end{displaymath}
lies entirely in the closure of the Riley slice and meets the boundary at a point $\mu$ corresponding to a  cusp group  where $P_{p/q}(\mu)=-2$.  These cusp groups have a limit set consisting of a
circle packing; two examples are depicted in Figure~\ref{fig:circle_packings} and the approximate positions of low-order cusp points are given in Table~\ref{tab:farey_words}. A result of
McMullen \cite{McMCusps} shows these limits to be dense in the boundary of the Riley slice.

\begin{figure}
  \centering
  \includegraphics[width=\textwidth]{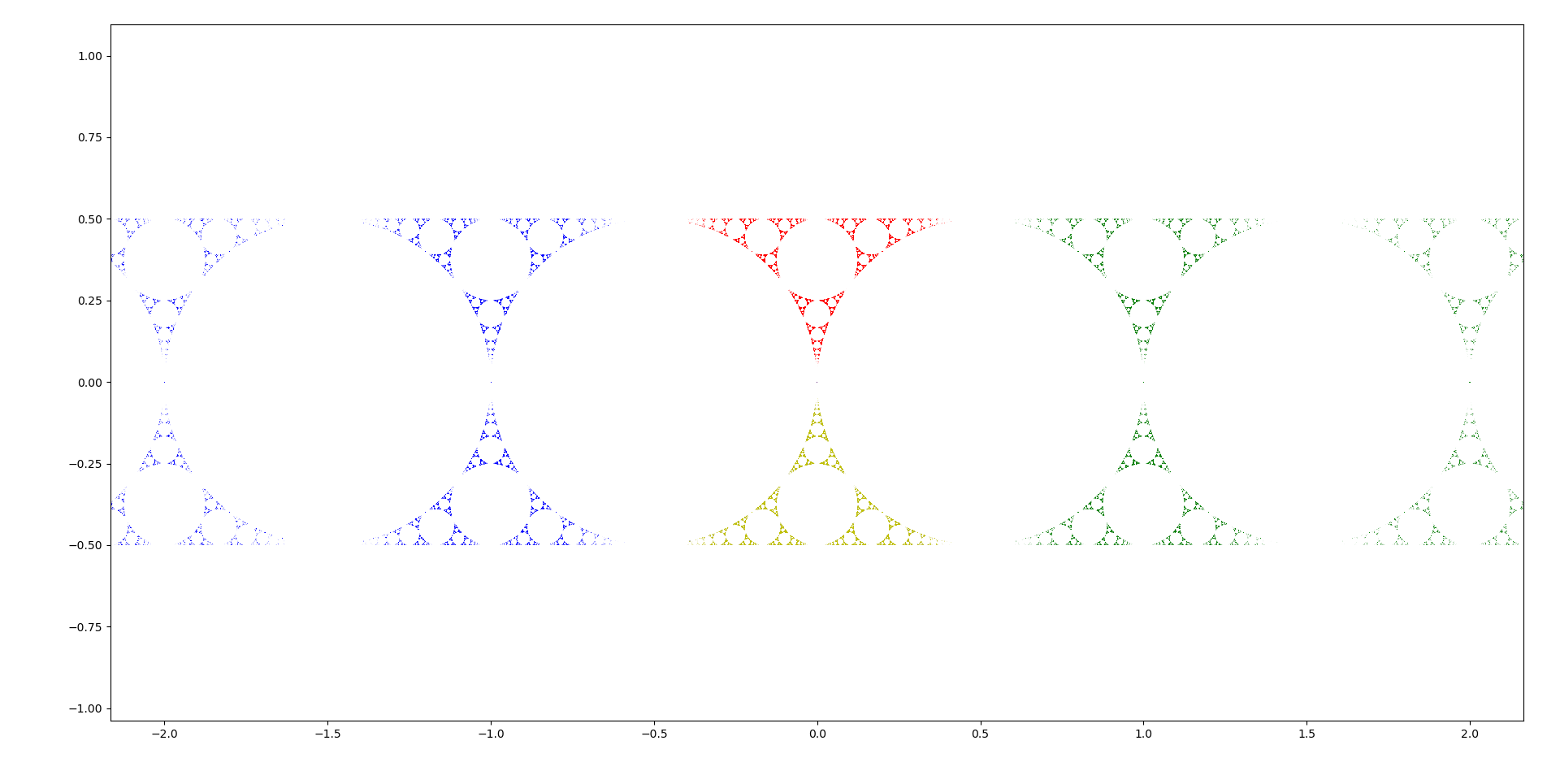}\\[0.5em]
  \includegraphics[width=\textwidth]{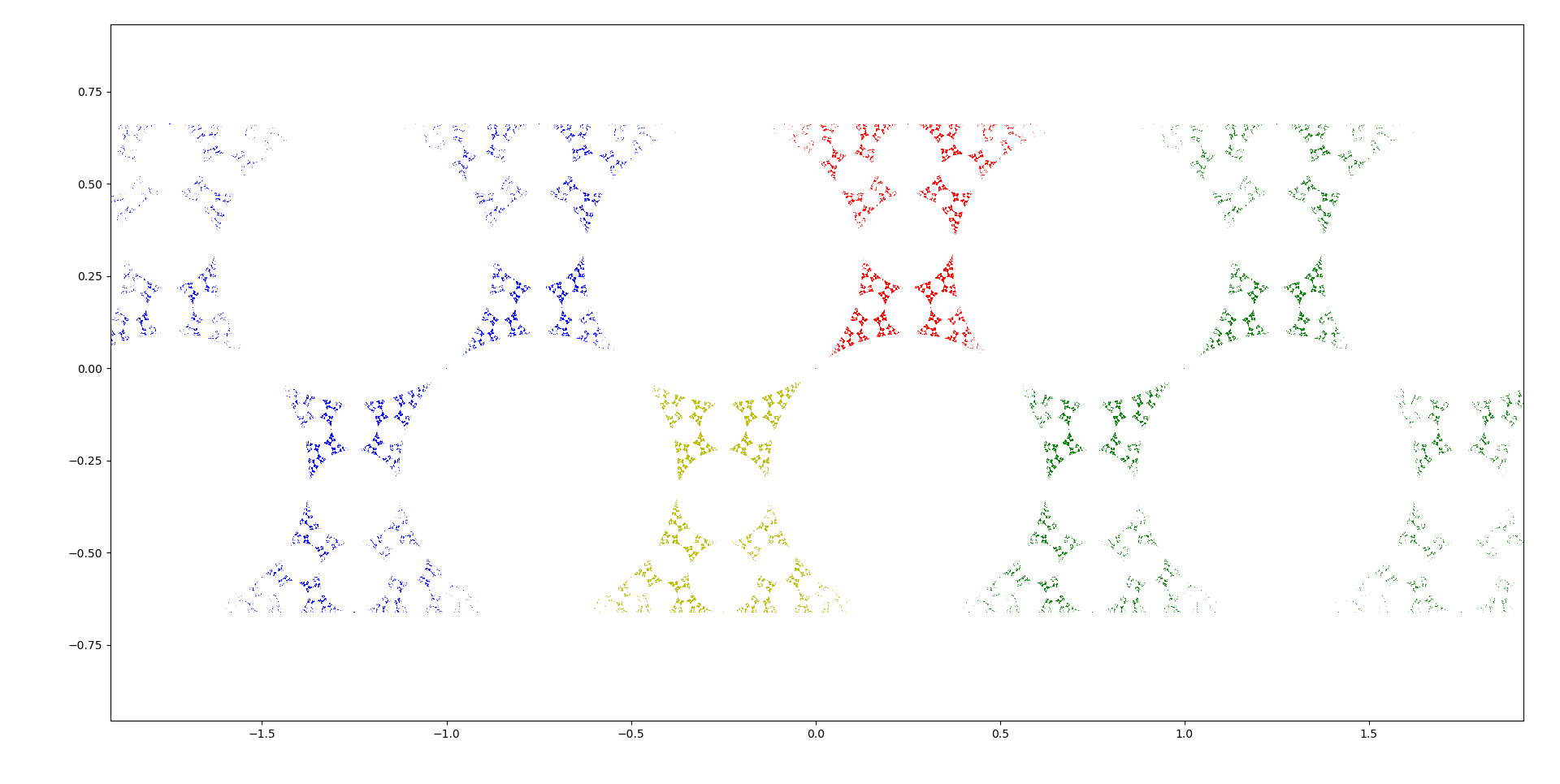}
  \caption{Circle packing limit sets of a cusp group. Top $1/2$,  bottom $2/3$. \label{fig:circle_packings}}
\end{figure}

\section{Farey polynomials}\label{sec:FP}
In this section, we prove various algebraic and dynamical results of the Farey polynomials.

\subsection{Discrete groups which lie on pleating ray extensions}
We begin with three elementary lemmata.
\begin{lemma}\label{lemma2}
  Let $\langle A,B\rangle$ be a subgroup of $\PSL(2,\IC)$ with $\tr^2 A = 4$, $\tr^2 B = 4$, $\tr(AB)-2=\nu\neq 0$, and such that neither $A$ nor $B$ is the identity.
  Then $\langle A,B \rangle$ is conjugate in $\PSL(2,\IC)$ to the group $\Gamma_\nu$.
\end{lemma}
\begin{proof}
  Let $f$ and $g$ be the M\"obius transformations of $\oC$ representing $A$ and $B$.  Then $f$ and $g$ are parabolic with fixed points $z_f$ and $z_g$.  Since $\mu\neq 0$, the
  mappings  $f$ and $g$ do not commute and $z_f\neq z_g$.  Choose a M\"obius transformation $h$ so that $h(z_f)=\infty$, $h(z_g)=0$, and $hfh^{-1}(0)=1$.  Then $hfh^{-1}(z)=z+1$
  and $hgh^{-1}(z)=z/(\alpha z+1)$.  We compute that
  \begin{displaymath}
    2+\alpha = \tr hfh^{-1} hgh^{-1} = \tr fg = \tr(AB) = 2+\nu
  \end{displaymath}
  and the result follows.
\end{proof}

\begin{lemma}\label{lemma3}
  Let $\Gamma_\mu$ be discrete and $\mu\neq 0$.  Then for all rational slopes $p/q $,
  \begin{displaymath}
    \abs{Q_{p/q}(\mu)} \geq 1
  \end{displaymath}
  unless $P_{p/q}(\mu)=2$.  This estimate is sharp.
\end{lemma}
\begin{proof}
  Write $W_{p/q}=\begin{pmatrix} a & b\\ c&d \end{pmatrix} $. Suppose first that $ c \neq 0 $. Then the Shimitzu--Leutbecher inequality \cite{Leut} applied
  to the discrete group $\langle f,h_{p/q}\rangle$ gives
  \begin{displaymath}
    1\leq \tr[f,h_{p/q}]-2 = \abs{c}^2=\abs{a+d-2}^2
  \end{displaymath}
  which is the desired result by Theorem~\ref{thmiso}. If $c=0$,  then $h_{p/q}$ is parabolic and also fixes $\infty$.

  The figure-eight knot-complement group is $\Gamma_{\mu_0}$ with $\mu_{0} = \frac{1}{2}(1+i\sqrt{3})$. With $p/q=1/1$ this shows the inequality to be sharp,  while the relator in this group is the $3/5$-Farey word, and $h_{3/5}=\mathrm{Id}$,  and
  \begin{displaymath}
    P_{3/5}(\mu)-2=-\mu (1 - \mu + \mu ^2)^2
  \end{displaymath}
  so $P_{3/5}(\mu_0)=2$.
\end{proof}

We remark that the Schubert normal form of the figure of eight knot complement is $5/3$.  More generally the relator in the two bridge knot or link complement with Schubert normal form $q/p$  is the $p/q$-Farey word.  Next we recall the following elementary result \cite[Lemma 3.2]{KS}.

\begin{lemma}\label{fuchs}
  Let $\Gamma_\mu$ be discrete.  If $h_{p/q}$ is real,  then $\langle f,h_{p/q}\rangle $ is Fuchsian.
\end{lemma}
\begin{proof}
  The group $\langle f,f^{-1}h_{p/q} \rangle$ is generated by two parabolics whose product is hyperbolic.
\end{proof}

As a consequence of these results, we can characterise the traces of the Farey words for discrete groups on the pleating ray extensions.

\begin{theorem}\label{thm:discrete_gps_on_pleating_rays}
  Let $\Gamma_\mu$ be discrete and $\mu\neq 0$.  Let $p/q$ be a rational slope and $P_{p/q}(\mu)\in \IR$.  Then either
  \begin{enumerate}
    \item $P_{p/q}(\mu) \geq 6$ or $P_{p/q}(\mu)=2+4\cos^2\big(\frac{\pi}{r}\big)$, where $r\geq 3$, or
    \item $P_{p/q}(\mu) \leq -2$ or $P_{p/q}(\mu)=2-4\cos^2\big(\frac{\pi}{r}\big) = -2\cos\big(\frac{2\pi}{r}\big)$, where $r\geq 3$.
  \end{enumerate}
  In particular,  on the extension of the pleating ray $\mathcal{P}_{p/q}$,  that is $P_{p/q}^{-1}((-\infty,0])$, the only allowable values for a discrete group are
  \begin{displaymath}
    P_{p/q}(\mu)= -2\cos(\frac{2\pi}{r}), \quad r\geq 3
  \end{displaymath}
  Each of these values occurs.
\end{theorem}
\begin{proof}
  The M\"obius transformation $f^{-1}h_{p/q}$ is parabolic.  There is an involution $\Phi$ conjugating $f^{-1}$ to $f^{-1}h_{p/q}$.  The group $\langle f,h_{p/q}\rangle$ is at most index two in $\langle \Phi,f\rangle$ and hence the latter is discrete. If $P_{p/q}(\mu)-2$ is real,  then $\langle f,h_{p/q}\rangle$ is Fuchsian by Lemma~\ref{fuchs}.  We have (in the notation of \cite{GGM})
  \begin{align*}
  \gamma(f,\phi) & = \tr [f,\phi]-2=\tr (f \Phi f^{-1} \Phi^{-1})-2=\tr (f \Phi f^{-1} \Phi^{-1})-2 \\
                 & = \tr (f f^{-1}h_{p/q}) -2 =  \tr h_{p/q} -2 = P_{p/q}(\mu)-2.
  \end{align*}
  and also
  \begin{displaymath}
    \beta(f)=\tr^2(f)-4=0,\quad \beta(\phi)=-4.
  \end{displaymath}
  Then with the assumption that $\gamma(f,\phi)\neq 0$ the discussion following (4.9) of \cite[Theorem 4.5]{GGM} tells us that $\langle f,\Phi \rangle$ is discrete if and only if either
  \begin{itemize}
    \item $ \gamma(f,\phi) \geq 4$,  or $\gamma(f,\Phi)=4\cos^2{\frac{\pi}{r}}$, $r\geq 3$,  or
    \item $ \gamma(f,\phi) \leq -4$,  or $\gamma(f,\Phi)=-4\cos^2{\frac{\pi}{r}}$, $r\geq 3$.
  \end{itemize}
  This is the statement of the theorem.

  The papers \cite[\S 2.]{A1} and \cite{A2} identify the group $\langle X,Y \mid (W_{p/q})^{r}=1\rangle$ in the complement of the Riley slice and as $\mu$ determines the group uniquely up to conjugacy the result follows.
\end{proof}

\subsection{Dynamical properties}\label{sec:dynamics}
We first give a short overview of the dynamical systems terminology which we shall use (following for example \cite{silverman}). A \textit{dynamical system} is a set $ S $ (the \textit{stable region})
together with a set $ \Phi $ of functions $ \phi : S \to S $ closed under iteration; in our case we will actually have an entire semigroup $ \mathcal{Q} $ of such functions. If $ S $ is a
metric space, the \textit{Fatou set} of a dynamical system is the maximal open set of $ S $ on which the functions of $ \Phi $ are equicontinuous; the \textit{Julia set} is the complement
of the Fatou set.\footnote{The Fatou set is analogous to the ordinary set of a Kleinian group, while the Julia set is analogous to the limit set.} The Julia set is often `thin' and so we want
to `thicken' it by `filling in the interior'. This motivates the definition of the \textit{filled Julia set} which has the Julia set as boundary: namely, if $ S $ happens to be a field $ k $, with complete
metric coming from some absolute value $ \abs{\cdot}_v $, we define the filled Julia set of $ \Phi $ to be
\begin{displaymath}
  \mathcal{K}(\Phi) := \{ x \in k : \sup_{\phi \in \Phi} \abs{\phi(x)} < \infty \}.
\end{displaymath}
(Of course in this paper we are working only over $ \mathbb{C} $, so all of this makes sense.) The complement of $ \mathcal{K}(\Phi) $ is the \textit{attracting basin of $ \infty $}. Finally, suppose $ \phi $ is
a rational function over $ \IC $ and let $ x $ be a fixed point of $ \alpha $. Then we variously say that $ x $ is
\begin{itemize}
  \item \textit{superattracting} if $ \phi'(x) = 0 $,
  \item \textit{attracting} if $ \abs{\phi'(x)} < 1 $,
  \item \textit{neutral} if $ \abs{\phi'(x)} = 1 $, and
  \item \textit{repelling} if $ \abs{\phi'(x)} > 1 $.
\end{itemize}

With all this in mind, our first main theorem of the paper is the following result.
\begin{theorem} \label{thm:rec}
  For each rational slope $p/q$ we have $Q_{p/q}(\mathcal{R})\subset \mathcal{R}$.
\end{theorem}
\begin{proof}
  Let $\Gamma_\mu=\langle f,g \rangle$ and let $h_{p/q}$ be the transformation corresponding to the Farey word $ W_{p/q}(\mu) $.  Consider the group
  \begin{displaymath}
    \tilde{\Gamma}  = \langle f,h_{p/q} \rangle = \langle f,f^{-1} h_{p/q} \rangle
  \end{displaymath}
  This group is generated by two parabolics by Lemma~\ref{lemma1}. Thus $\tilde{\Gamma}$ is a conjugate of $\Gamma_\nu$ where
  \begin{displaymath}
    \nu=\tr( f f^{-1} h_{p/q})-2  =  P_{p/q}(\mu)-2
  \end{displaymath}
  by Lemma~\ref{lemma2}.  As a conjugate of a  subgroup of $\langle f,g \rangle$ the group $\Gamma_\mu$ is discrete.  It is also free with nonempty ordinary set,
  moreover $\Lambda(\tilde{\Gamma})$ is a subset of $\Lambda(\Gamma_\mu)$ and the M\"obius image of $\Lambda(\Gamma_\nu)$.    Hence $\nu\in \mathcal{R}$.
\end{proof}

\begin{corollary}
  For each rational slope $p/q$ we have that the algebraic set
  \begin{displaymath}
    \mathbf{Z}(Q_{p/q}) = \{z\in \IC: Q_{p/q}(z)=0 \}
  \end{displaymath}
  is contained within the exterior of the closure of the Riley slice, i.e. within $ \IC\setminus \overline{\mathcal{R}} $.
\end{corollary}
Of course more is true here.
\begin{corollary}
  Let $z_0\not\in\overline{\mathcal{R}}$. For each rational slope $p/q$ we have
  \begin{displaymath}
    \{z\in \IC:Q_{p/q}(z)=z_0 \}  \subset \IC\setminus \overline{\mathcal{R}}.
  \end{displaymath}
\end{corollary}

The semigroup generated by the polynomials
\begin{displaymath}
  {\cal Q}=\langle Q_{p/q}: \mbox{ $p/q$ is a rational slope} \rangle
\end{displaymath}
with the operation of functional composition now sets up a dynamical system on $\oC$ for which $\mathcal{R}$ lies in the stable region.  No filled Julia set
for any polynomial $Q_{p/q}$ can meet $\mathcal{R}$; for three examples, see Figure~\ref{fig:julias}.  Equivalently, $\mathcal{R}$
lies in the superattracting basin of $\infty$ for every polynomial. Every polynomial $Q_{p/q}$ has zero as a fixed point, $Q_{p/q}(0)=0$.

\begin{figure}
  \centering
  \includegraphics[width=0.32\textwidth]{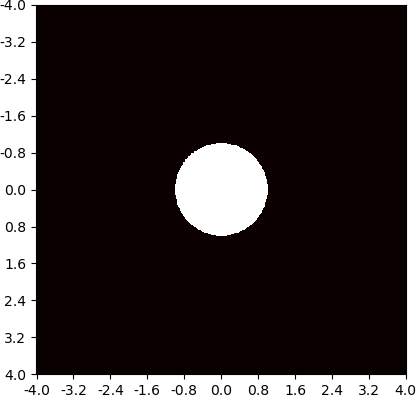}%
  \includegraphics[width=0.32\textwidth]{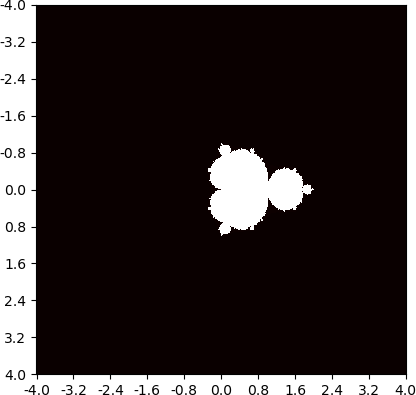}%
  \includegraphics[width=0.32\textwidth]{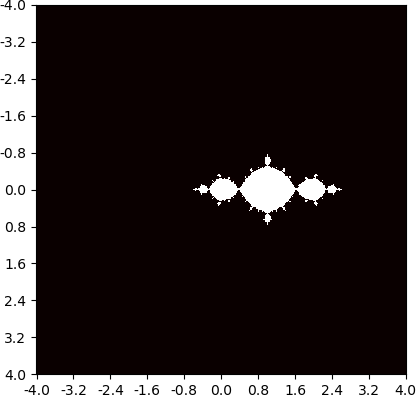}%
  \caption{Filled Julia sets for $ Q_{1/2} $ (left), $ Q_{1/3} $ (middle), and $ Q_{1/4} $ (right).\label{fig:julias}}
\end{figure}

Our computational evidence suggests the following conjecture:

\begin{conjecture}\label{conj}
  If $p/q$ is a rational slope,  then $Q_{p/q}$ factors as
   \begin{itemize}
    \item $Q_{p/q}(z) = \pm z u(z)^2$,  $u(0)=1$, when $q$ is odd.
    \item $Q_{p/q}(z) =  z^{2[(n+1)/2]} v(z)^2$,   where $q=2^n r$, $r$  odd.
   \end{itemize}
\end{conjecture}

We wish to explore this a little further.  The following lemma is a simple consequence of the form of a Farey word.
\begin{lemma}
  As $\mu\to 0$ we have
  \begin{itemize}
    \item $W_{p/q}(\mu) \to X$ if $q$ is odd.
    \item $W_{p/q}(\mu) \to \mathrm{Id} $ if $q$ is even. \qed
  \end{itemize}
\end{lemma}

From this lemma, we easily classify the fixed point type of $ 0 $.
\begin{theorem}  \label{thm:superattractor}
  If $q$ is even, then $0$ is a superattracting fixed point for $Q_{p/q}$. If $ q $ is odd, then $ 0 $ is not superattracting.
\end{theorem}
\begin{proof}
  Since $ \det W_{p/q} = 1 $, by Theorem~\ref{thmiso} we have $ a+d-2 = c = Q_{p/q} $. Substituting $ ad - bc = 1 $ we obtain $ c(a-b) = (a-1)^2 $. If $ q $ is even,
  then $ a(\mu) \to 1 $ as $ \mu \to 0 $ so the right hand side of this equality is a polynomial with a double root at 0. On the left, $ a(\mu)-b(\mu) \to 1 $
  as $ \mu \to 0 $; in particular, $ (a-b)(0) \neq 0 $. Thus both roots at $ 0 $ must come from factors of $ c(\mu) $, i.e. $ \mu^2 \mid c(\mu) $ and so $ c'(0) = 0 $.

  For $ q $ odd, we have again $ c(a-b) = (a-1)^2 $; again the right side has a double root at 0, but on the left we have a factor $ (a-b) $ which becomes 0 at 0; since $c$
  also has a root at 0, it follows that both $ (a-b) $ and $ c $ have single roots at 0.
\end{proof}
\begin{remark}
  Of course it would follow easily from Conjecture~\ref{conj} that $ 0 $ is a \emph{neutral} fixed point.
\end{remark}

In \cite{M1} it is shown that the semigroup generated by all word polynomials has the complement of the Riley slice as its Julia set.  As such, the roots of the word polynomials are
dense in $\IC\setminus \mathcal{R}$ as backward orbits are dense in the Julia set. In the context of Farey words this suggests that the roots of all compositions of Farey words are dense.
However somewhat more appears true ---  See Figure~\ref{fig:TwoRileySlice}.

These pictures are quite quickly generated and give a good approximation to the Riley slice even for much smaller bounds on the denominators $q$.  Notice that in view of Lemma~\ref{lemma3}
there is an open pre-image of the unit disk about each point also lying in $\IC\setminus \overline{\mathcal{R}}$.

\begin{figure}
  \centering
  \includegraphics[width=\textwidth]{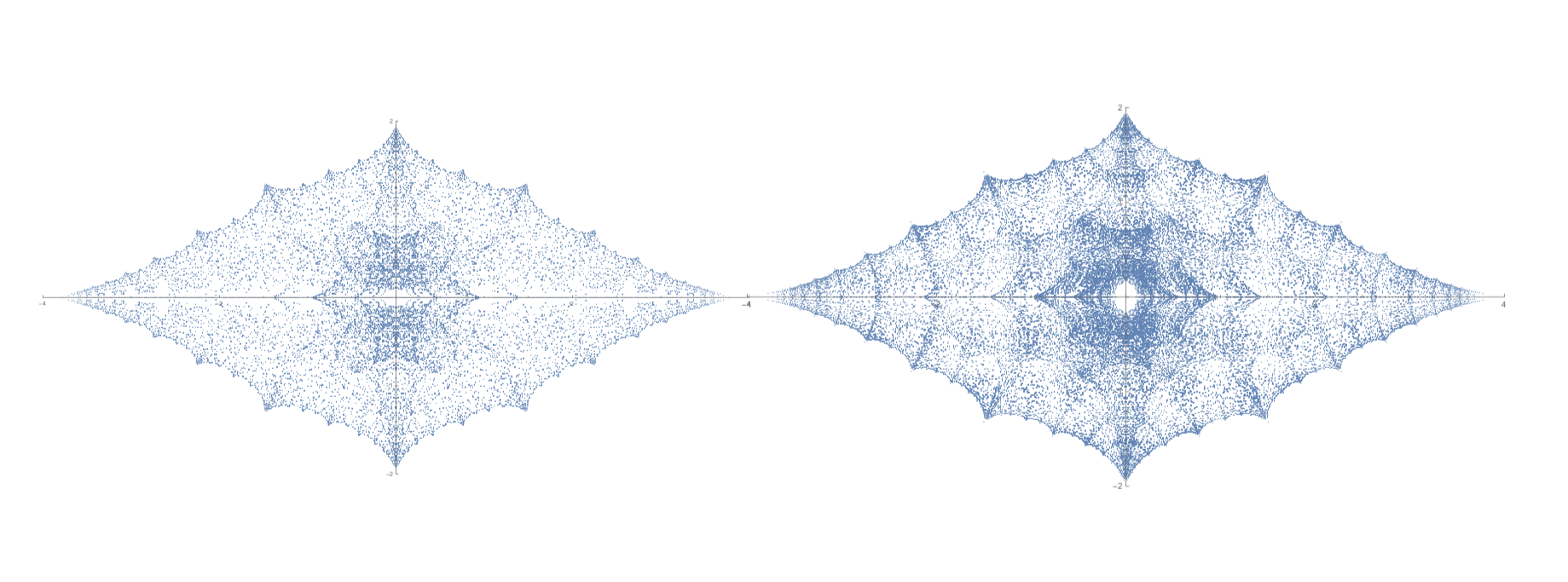}
  \caption{For all rational slopes with $q\leq 377$,  Left: Root set of the polynomials $Q_{p/q}$. This set lies in $\IC\setminus \overline{\mathcal{R}}$. Right: Root set of the equation $Q_{p/q}=-4$. This latter set lies in $\IC\setminus \mathcal{R}$ and contains all the cusp groups.  Notice the appearance of clustering around the pleating rays.\label{fig:TwoRileySlice}}
\end{figure}

\section{Neighbourhoods of rational pleating rays}\label{sec:NRPR}

In this section, we give some motivation and intuition for our second main result. Our first main result gave a method of approximating the Riley slice exterior using the Farey polynomials and some related Julia
sets; our second result is an approximation of the \emph{interior} using the Farey polynomials. Here is the precise statement:
\begin{theorem}\label{main}
  Let $P_{p/q}$ be a Farey polynomial.  Then there is a branch of the inverse of $ P_{p/q} $ such that
  \begin{displaymath}
    P_{p/q}^{-1}({\cal H}_{-2})\subset \mathcal{R}, \quad {\cal H}_{-2}=\{ \Re(z) < -2\}.
  \end{displaymath}
\end{theorem}

The bounds given in the theorem are illustrated in Figure~\ref{fig:RileyPleatedNbhds}.

\begin{figure}
  \centering
  \includegraphics[width=\textwidth]{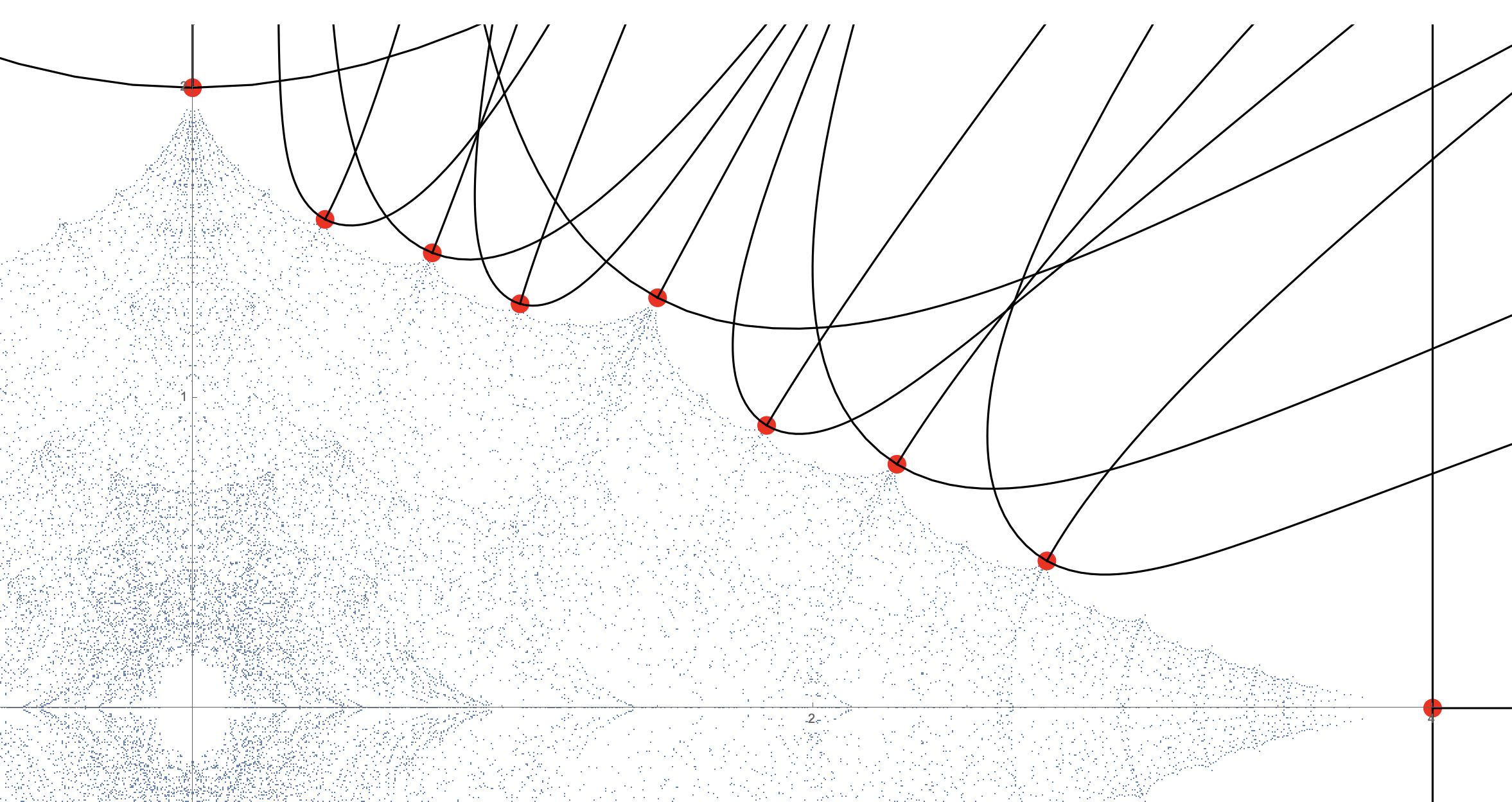}
  \caption{The Riley slice with neighbourhoods for our pleating ray values illustrated.\label{fig:RileyPleatedNbhds}}
\end{figure}

\subsection{Motivating remarks}
We make the following remarks that heuristically suggest why Theorem~\ref{main}  might be true and which underpin our proof.  The Riley slice $\mathcal{R}$ is topologically a punctured disk in the
plane and as such admits a hyperbolic metric which we denote $\rho_\mathcal{R}:\mathcal{R}\times\mathcal{R}\to [0,\infty)$.

\begin{theorem}\label{thm3}
  Let $\alpha$ be a curve in $\mathcal{R}$ which lies a bounded hyperbolic distance from a pleating ray (that is, there exists an $M<\infty$ such that for each $\mu_0\in\alpha$ there
  is $\nu_0\in \mathcal{P}_{p/q}$ with $\rho_\mathcal{R}(\mu_0,\nu_0)\leq M$).  Then each $\Gamma_{\mu_0}$ is quasiconformally conjugate to $\Gamma_{\nu_0}$ by a quasiconformal mapping
  of $\IC$ with distortion no more than $e^M$.
\end{theorem}
\begin{proof}
  Let $\mu_0\in\alpha$ and $\nu_0\in \mathcal{P}_{p/q}$ with $\rho_\mathcal{R}(\mu_0,\nu_0)=R$.  Let $\Phi:\ID\to \mathcal{R}$ be the universal cover map (inducing the hyperbolic metric) with $\Phi(0)=\mu_0$ and $\Phi(\tanh(e^{R/2}))=\nu_0$.  The holomorphically parameterised family of discrete groups $\{ \Gamma_{\Phi(z)}:z\in \ID \}$ induces an equivariant ambient isotopy of $\oC$ by the Sullivan-Thurston theory of holomorphic motions,  in particular Slodkowski's equivariant version \cite{Slod1,Slod2,EKK}.  Roughly if we move $\mu$ in $\mathcal{R}$,  then as solutions to polynomial equations, the fixed point sets of elements of the group $\Gamma_\mu$ move holomorphically and do not collide while the group remains discrete and until the formation of ``new'' parabolic elements.  These fixed point sets are dense in the limit set and their motion extends to a holomorphically parameterised quasiconformal ambient isotopy,  equivariant with respect to the groups $\Gamma_\mu$,  of the whole complex plane. The distortion of this ambient isotopy is no more than the exponential of the hyperbolic distance between the start (at $0$) and finish (at $\tanh(e^{R/2})$),  that is $e^{R}$.
\end{proof}

In their paper, Keen and Series show that as we move down a pleating ray (and shrinking the hyperbolic translation length of a Farey word $W_{p/q}$ and the length of a
simple closed curve on $S^2_4$) there is a natural combinatorial pattern of round disks which they call \textit{F-peripheral disks}, stabilised by the Fuchsian group $\langle f,h_{p/q}\rangle$
and closely related to the peripheral subgroups of the fundamental group of the $3$-manifold $\IH^3/\Gamma$.  There is a non-conjugate pair of these peripheral disks. They both contain
a conjugate of the Farey word in their stabiliser and these peripheral disks persist in small deformations \cite[Proposition 3.1]{KS}, and force the quotient $(\oC\setminus \Lambda_\mu)/\Gamma_\mu$
to be the four-times punctured sphere \cite[Lemma 3.5]{KS}.  This process continues until the Farey word becomes parabolic.  This pinching then forces circles in the limit set to become tangent,
and the quotient of the ordinary set to degenerate to a disjoint union of two three-times punctured spheres.

Consider deforming a point $\mu \in \mathcal{R}$ towards the Riley slice boundary along a curve $\alpha$ which lies a bounded distance away from a pleating ray $ \mathcal{P}_{p/q} $. Theorem~\ref{thm3} shows
that if $\nu \in \mathcal{P}_{p/q} $ is the hyperbolically closest point to $\mu $ on the pleating ray then the combinatorial properties of circles in the limit set of $ \Gamma_{\nu} $ transfer directly
to combinatorial properties of quasicircles in the limit set of $ \Gamma_{\mu} $ since there is a uniformly bounded distortion mapping one to the other. These quasicircles bound what we will call
the \textit{peripheral quasidisks} of the group $ \Gamma_{\nu} $ (by analogy with the theory of Keen and Series).

Most of the information that the Keen--Series theory provides is topological and their arguments could be used almost directly if we knew these uniform bounds. However,  there is no way that we can
compute or even estimate the hyperbolic metric of the Riley slice near the boundary to identify a curve such as $\alpha$ for every rational pleating ray.  What we do is guess (motivated by examining
a lot of examples) that such a curve is $\alpha=P_{p/q}^{-1}(\{z=-2+i t:t>0\}$, where we take the branch of the inverse of $ P_{p/q} $ with the correct asymptotic behaviour.

Indeed it is easily seen that this curve works for the identity word $P_{1/1}(\mu)=2+\mu$.  Early pictures of $\mathcal{R}$ by Riley suggest there is a cusp of $\mathcal{R}$ at $-4$.  In the next section we recall the main
result of a paper of Lyndon and Ullman \cite{LU} and examine it in this context.

\subsection{Lyndon and Ullman's results}

\begin{theorem}[Theorem 3, \cite{LU}]\label{LU}
  Let $K$ denote the Euclidean convex hull of the set $\ID(0,2)\cup\{\pm 4\}$.  Then $\IC\setminus \mathcal{R} \subset K$.
\end{theorem}
See Figure~\ref{fig:LURegion} for a depiction of this bound.

\begin{figure}
  \centering
  \includegraphics[scale=0.2]{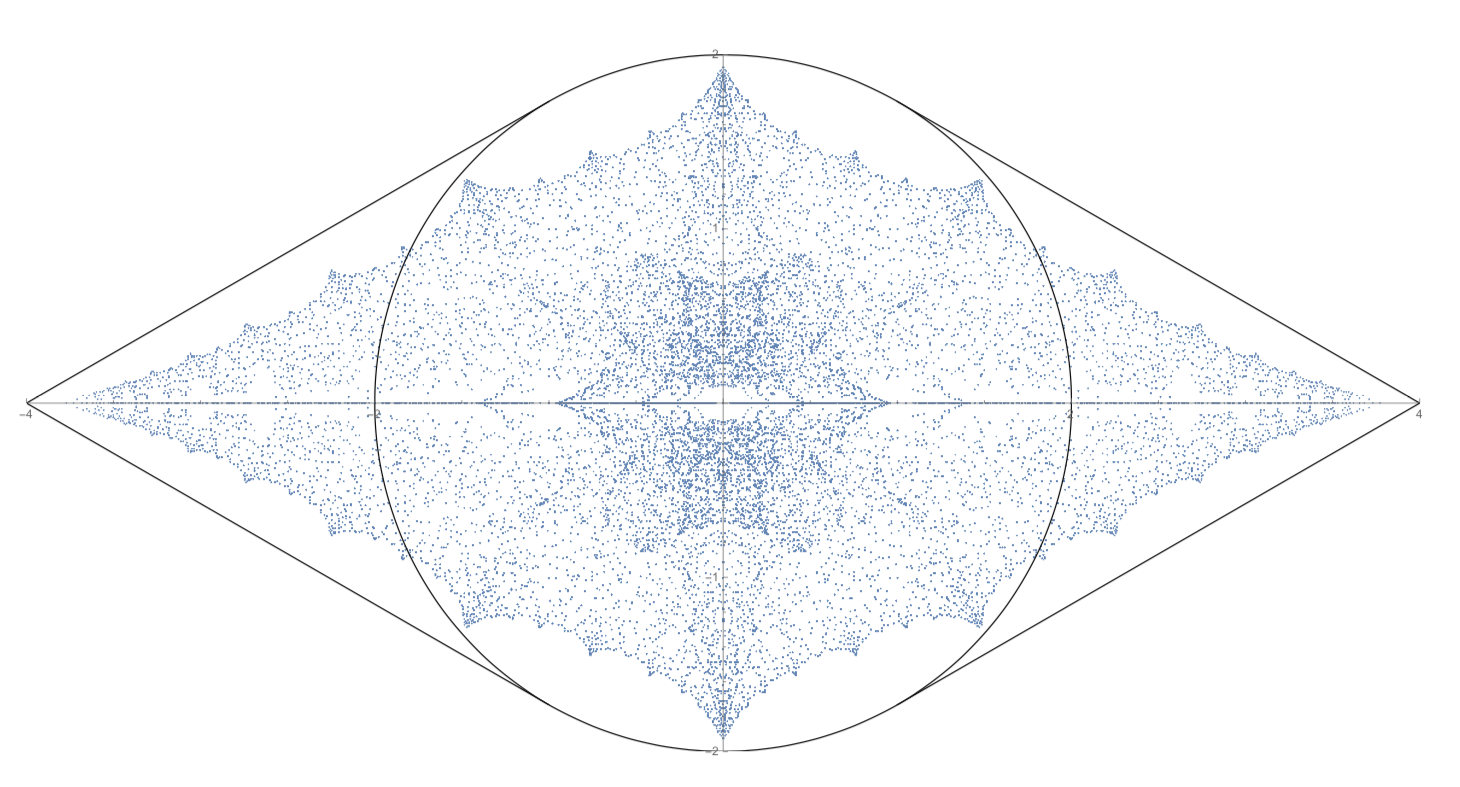}
  \caption{The convex hull of $\ID(0,2)\cup\{\pm 4\} $ contains $\IC\setminus \mathcal{R}$.\label{fig:LURegion}}
\end{figure}

Let ${\cal W}=\{z:\frac{-\pi}{6}< \arg(z+4)<\frac{\pi}{6}\}$ denote the sector of solid angle $\pi/3$ with tip at $-4$.  Let $\varphi(z)$ be the branch of $z\mapsto -(-z-4)^{3/5}-4$
conformally mapping $\IC\setminus {\cal W}$ to the half-space $H=\{z:\Re(z)<-4\}$.  Then  $H\subset {\cal W}$ and $\varphi(H)$ is the sector $\{z:\frac{5\pi}{6}< \arg(z+4)<\frac{7\pi}{6}\}$.
Because $\varphi$ is conformal it is now straightforward to see that the distance in the hyperbolic metric of ${\cal W}$ between the line $\ell_1=-4+i\IR$ and the rational pleating ray $\ell_2=(-\infty,-4]$ is
\begin{displaymath}
  \rho_{{\cal W}}(\ell_1,\ell_2) = \int_{0}^{3\pi/10} \frac{d\theta}{\cos(\theta)} = \frac{1}{2}\ln\left[5+2 \sqrt{5}\right] \approx 1.1241.
\end{displaymath}
From Theorem~\ref{thm3} we now have the following corollary.

\begin{corollary}
  Let $\nu\in -4+i\IR$.  Then there is $\mu\in (-\infty,-4]$,  the rational pleating ray $1/2$,  so that $\Gamma_\nu$ and $\Gamma_\mu$ are $K$-quasiconformally conjugate and
  \begin{displaymath}
    K \leq \sqrt{5+2 \sqrt{5}} \approx 3.077\ldots
  \end{displaymath}
\end{corollary}
\begin{proof}
  The only thing left to observe is that the contraction principle for the hyperbolic metric (really Schwarz' lemma in disguise) shows that the hyperbolic metric of $\mathcal{R}$ is
  smaller than the hyperbolic metric of $\IC \setminus {\cal W}$ so that
  \begin{displaymath}
    \rho_\mathcal{R}(\mu,\nu)\leq \rho_{\cal W}(\mu,\nu) = \frac{1}{2}\ln\left[5+2 \sqrt{5}\right]
  \end{displaymath}
  for the point $\mu$ closest to $\nu$, and hence by Theorem~\ref{thm3}
  \begin{displaymath}
    K\leq e^{\rho_\mathcal{R}(\mu,\nu)} \leq \sqrt{5+2 \sqrt{5}}.
  \end{displaymath}
  This proves the corollary.
\end{proof}

\begin{figure}
  \centering
  \includegraphics[scale=0.3]{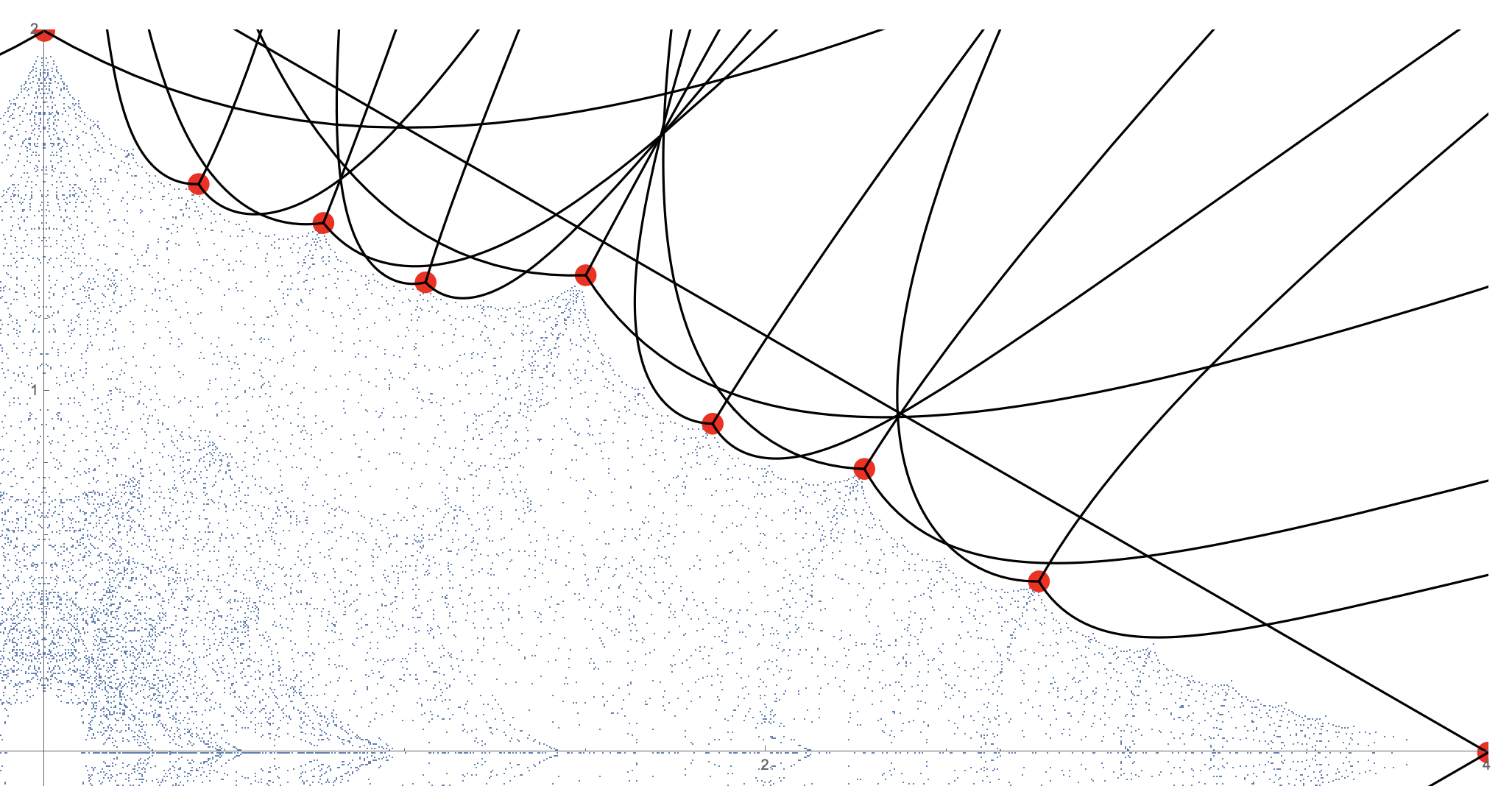}
  \caption{Preimages of the sector $\{\arg(z) = -\frac{5\pi}{6}\}\cup\{\arg(z)=\frac{\pi}{6}\}$.\label{fig:LURays}}
\end{figure}

We believe these estimates persist in what we will prove but getting them adds additional complications in the construction we give when the isometric circles of $W_{p/q}$ are no longer
disjoint.   We offer Figure~\ref{fig:LURays}, which is a slight modification of Figure~\ref{fig:RileyPleatedNbhds} (p.~\pageref{fig:RileyPleatedNbhds}), as computational support for this conjecture.
Instead of looking at the branch of the inverse of $P_{p/q}$ defined on $\{ \Re(z)< -4 \} $, to produce this image we compute the preimages of the conic region of opening $\frac{\pi}{3}$ given by Lyndon and Ullman.

\section{Proof of Theorem~\ref{main}}\label{sec:PF}

Our proof is structured as follows, closely mimicking that of Keen and Series. We start on a rational pleating ray at a value $\mu_0\in \mathcal{R}$ and move $\mu$ away from it.
Since $\mathcal{R}$ --- in fact $\overline{\mathcal{R}}$ (by Lemma~\ref{lem:bdry_discrete}) --- consists of discrete groups,  discreteness will never be an issue for us.  For a
small variation the  groups $\Gamma_\mu=\langle f,g\rangle$ are discrete Schottky groups with quotient the four-times punctured sphere.  The key issue is the open/closed
argument in the proof of \cite[Theorem 3.7]{KS}.  Openness will be directly as they argue, but without control on the distortion of the induced combinatorial pattern the peripheral
quasicircles can become quite entangled and eventually become space filling curves. This is the situation we must avoid and we do it by modifying the peripheral quasidisks as we move,
so they have large scale ``bounded geometry'' (though the small scale geometry is uncontrolled).  An important observation is that along the rational pleating ray the isometric circles
of the Farey word $h_{p/q}$ are disjoint.  We move away keeping this condition.  Further,  if we do not move too far away these isometric circles do not start spinning around one
another.  This information allows us to construct a ``nice'' precisely invariant set stabilised  by $f$ and $h_{p/q}$ --- one of the peripheral quasidisks with bounded geometry.
Existence of this peripheral quasidisk guarantees we have quotient $S^2_4$ from the action of $\Gamma_\mu$ on the ordinary set. These nice configurations persist with an open and
(relatively) closed argument within a certain region and so we remain in the Riley slice through this deformation.

It may be useful to have a reference to a specific example.  In Figure~\ref{fig:CuspGroup} are pictures of the geometric objects we will be interested in for two specific cusp groups.

\begin{figure}
  \centering
  \includegraphics[width=\textwidth]{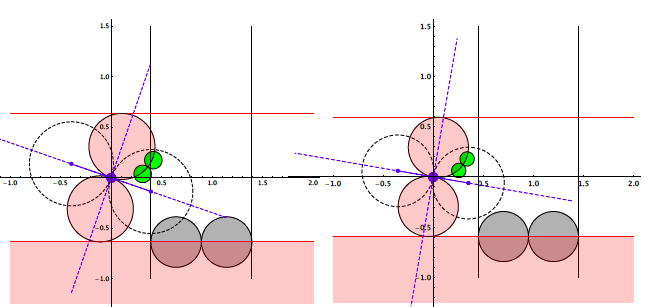}
  \caption{Geometric objects in the limit sets of the $3/4$-cusp group (left) and the $4/5$-cusp group (right).  The isometric circles of $W_{p/q}$ are shaded grey; their images under the involution $\Phi:z\mapsto1/(\mu z)$ are shaded green.  This involution defines the non-conjugate peripheral disk $\langle Y, \Phi W_{p/q} \Phi^{-1}\rangle$.  The non-conjugate peripheral disks are shaded in red (one is a lower half plane).  Fixed points of the involution and its action are also illustrated.\label{fig:CuspGroup}}
\end{figure}

\subsection{Products of parabolics}
As noted earlier (Lemma~\ref{lemma1}), an important property of a Farey word $W_{p/q}$  is that it can be written as a product of parabolic elements in two essentially different ways.
For $\mu\in \mathcal{R}$ there are only two conjugacy classes of parabolics, those represented by $X$ and $Y$ \cite[VI.A]{Mas}.  This is explained in \cite[\S2]{KS}; it is just a reflection
of the fact that a simple closed curve on the four-times punctured sphere must separate one pair of punctures from another,  so the deletion of this curve leaves two doubly punctured
disks. To find these parabolics we just look for a couple of conjugates  of $X$ and $Y$ whose product is $W_{p/q}$. Keen and Series studied the set of all such pairs (this is the data encoded in
the subgroups $ \mathcal{U}_{p/q} $ of their paper).  However, we will really only look closely at the pair $\{X, X^{-1}W_ {p/q}\}$. The group $\langle X,W_ {p/q} \rangle = \langle X,X^{-1}W_ {p/q} \rangle$
is generated by two parabolics. This group can therefore only be discrete and free on its generators if $\tr(X .X^{-1}W_ {p/q})-2=\tr(W_ {p/q})-2 \in \overline{\mathcal{R}}$.  If $\tr(W_ {p/q})\in \IR$,
then the traces of $X$, $X^{-1}W_ {p/q}$ and $W_ {p/q}$ are real (the first two are $\pm 2$) and so $\langle X,W_ {p/q} \rangle$ is Fuchsian.  These groups and their conjugates are where the round $F$-peripheral
circles in \cite{KS} come from.

\begin{lemma}
  Suppose that $\Gamma_u=\langle u_1,u_2 \rangle$ and $\Gamma_v=\langle v_1,v_2 \rangle$ are two groups generated by parabolics $ u_1,u_2,v_1,v_2 $,  and that
  $\tr(u_1u_2)=\tr(v_1v_2)$.  Then $\Gamma_u$ and $\Gamma_v$ are conjugate in $\PSL(2,\IC)$.
\end{lemma}
\begin{proof}
  A little more is true.  There is an involution $\phi_u$ so that $\phi_u u_1 \phi_u^{-1}=u_2$.  We can similarly define an involution $\phi_v$.  Then, with $\beta(w)=\tr^2(w)-4$,
  \begin{displaymath}
    \beta(\phi_u)=\beta(\phi_v)=-4, \;\; \beta(u_1)=\beta(v_1)=0.
  \end{displaymath}
  Also
  \begin{displaymath}
    \tr([u_1,\phi_u])=\tr( u_1\phi_u u_1^{-1}\phi_u^{-1})=\tr(u_1u_2)=\tr(v_1v_2)=\tr([v_1,\phi_v])
  \end{displaymath}
  In \cite{GM} it is shown that any pair of two-generator groups with the same trace square of the generators,  and same trace of the commutators are conjugate in $\PSL(2,\IC)$.
  Thus $\langle v_1,\phi_v \rangle$ and $\langle u_1,\phi_u \rangle$ are conjugate which implies the result we want.
\end{proof}

The upshot of this lemma is that if we were to pick a different pair whose product was $ W_{p/q} $ then we get exactly the same geometry, up to a well-defined conjugation in $ \hat\IC $.

\subsection{Holonomy and isometric disks}
Let $W\in \PSL(2,\IC)$ with $\tr(W)=-2+i t$.  Then the M\"obius transformation $f$ representing $W$ has translation length $\tau_f$ and holonomy $\theta_f$ where
\begin{displaymath}
  \frac{\tau_f}{2}= \Re\left[\sinh^{-1}\left(\frac{i}{2} \sqrt{t (4 i+t)}\right)\right], \; \frac{\theta_f}{2}= \Im\left[\sinh^{-1}\left(\frac{i}{2} \sqrt{t (4 i+t)}\right)\right]
\end{displaymath}
We also have the following asymptotics.
\begin{gather*}
  \frac{ \tau_f}{\sqrt{2t}} \to 1, \mbox{ as $t\to 0$.\quad For $0<t<1$, we have } 1 \leq \frac{ \tau_f}{\sqrt{2t}} \leq 1.03642\ldots\\
  \frac{\theta_f}{\sqrt{2t}} \to -1, \mbox{ as $t\to 0$.\quad For $0<t<1$, we have } -1  \leq \frac{ \theta_f}{\sqrt{2t}} \leq -0.954\ldots
\end{gather*}
and $ \theta_f \to -\pi$, as $t\to \infty$.

The isometric disks of $f(z)$ are the two disks.
\begin{displaymath}
  D_1=\left\{z: \big|z-\frac{a}{c}\big|\leq \frac{1}{|c|} \right\}, \quad D_2= \left\{z:\big|z+\frac{d}{c}\big|\leq \frac{1}{|c|}\right\}, \quad f \sim w=\left(\begin{array}{cc} a & b\\c & d \end{array} \right)
\end{displaymath}
The isometric circles are the boundaries of these two disks.  We say   that $f$ has \textit{disjoint isometric disks} if these disks are disjoint.  This is clearly equivalent to the condition $ |a+d|\geq 2 $.

The mapping $f$ pairs these disks in the sense that
\begin{displaymath}
  f(D_1)=\oC \setminus \overline{D_2}.
\end{displaymath}
Thus $\oC\setminus \overline{D_1\cup D_2}$ is a fundamental domain for the action of $f$ on $\oC$.
Notice that when $c\neq 0$,  $f(\infty)=\frac{a}{c}$ and that $f^{-1}(\infty)=-\frac{d}{c}$ are the centers of the isometric disks.

We now specialise to the case that the transformation is a Farey word. Recall from above the notation
\begin{displaymath}
  W_{p/q}(\mu) = \begin{pmatrix} a_{p/q}(\mu) & b_{p/q}(\mu) \\ c_{p/q}(\mu) & d_{p/q}(\mu) \end{pmatrix} \qquad  a_{p/q}d_{p/q}-b_{p/q}c_{p/q}=1.
\end{displaymath}

\begin{lemma}
  Let $\Re(\mu)\leq -2$.  Then the Farey word $W_{p/q}(\mu)$ has disjoint isometric disks.
\end{lemma}
\begin{proof}
  The isometric circles of $W_{p/q}$ are the two disks
  \begin{equation}\label{isofarey}
    D_1=\ID\Big(\frac{a_{p/q}(\mu)}{c_{p/q}(\mu)},\frac{1}{|c_{p/q}(\mu)|}\Big),  \quad D_2=\ID\Big(\frac{-d_{p/q}(\mu)}{c_{p/q}(\mu)},\frac{1}{|c_{p/q}(\mu)|}\Big)
  \end{equation}
  We now compute with the identity of Theorem~\ref{thmiso} that
  \begin{equation}\label{isoin}
    \frac{a_{p/q}(\mu)}{c_{p/q}(\mu)}+\frac{d_{p/q}(\mu)}{c_{p/q}(\mu)} = \frac{2+c_{p/q}(\mu)}{c_{p/q}(\mu)} = 1+\frac{2}{c_{p/q}(\mu)}
  \end{equation}

  Now $\tr(W_{p/q})=-2+it$ implies $|a_{p/q}+d_{p/q}|=\sqrt{4+t^2}>2$ so along the the path  $\tr(W_{p/q})=-2+it$ we have that the Farey word $W_{p/q}$ has disjoint isometric disks.
\end{proof}

Theorem~\ref{thmiso} and Equation~\ref{isoin})  together have the following consequence.
\begin{corollary}\label{cor1}
  Let $\tr(W_{p/q})=-x+it$, $x\geq 2$ and $t\in \IR$.  Then the group $\langle X,W_{p/q}\rangle$ is discrete and free on the indicated generators.
\end{corollary}
\begin{proof}
  Let $S$ be the vertical strip of width one between $\frac{1}{2}\Big(\Re\big(\frac{a(\mu)-d(\mu)}{c(\mu)}\big)-1\Big)$ and $\frac{1}{2}\Big(\Re\big(\frac{a(\mu)-d(\mu)}{c(\mu)}\big)+1\Big)$.  Let
  \begin{displaymath}
    D_1=\ID\Big(\frac{a(\mu)}{c(\mu)},\frac{1}{|c(\mu)|}\Big),  \quad D_2=\ID\Big(\frac{-d(\mu)}{c(\mu)},\frac{1}{|c(\mu)|}\Big)
  \end{displaymath}
  and $\tilde{D_1}=D_1-1$ and $\tilde{D_2}=D_2+1$.  Then the $ D_i $ are the isometric circles for $ W_{p/q} $, and each $ \tilde D_i $ is a translate of the respective $ D_i $ (to the left and right respectively; see Figure~\ref{fig:cor1_circles}). Define $ \tilde S $ by
  \begin{equation}\label{19}
    \tilde{S} = (S \cup D_1\cup D_2)\setminus \big( \tilde{D}_1\cup \tilde{D_2}\big).
  \end{equation}

  Theorem~\ref{thmiso} implies that the disks $ \tilde D_1 $ and $ D_2 $ are tangent.  Two things now follow.  First, the translates of $\tilde{S}$ by $n\in\IZ$ fill the plane.  Second, $\tilde{S}$ contains the isometric circles of $W_{p/q}$.  The Klein combination theorem \autocite[Theorem VII.A.13]{Mas} now implies the result since $\hat{\IC}\setminus (D_1\cup D_2)$ is a fundamental domain for the action of $W_{p/q}$.
\end{proof}
\begin{figure}
  \centering
  \includegraphics[width=0.8\textwidth]{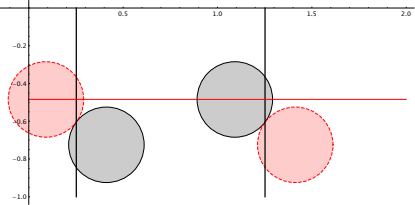}
  \caption{The isometric circles of $W_{3/4}$ when $\tr(W_{3/4})=-2+i$ (in grey) and their translates (in red). \label{fig:cor1_circles}}
\end{figure}

There is one further piece of information we would like out of Corollary~\ref{cor1}: that the point of tangency of the isometric disks and their translates is a parabolic fixed point.  Let $h=h_{p/q}$ represent $W_{p/q}$.  This point of tangency can be calculated to be
\begin{displaymath}
  \frac{a-d}{2c}-\frac{1}{2} =  \frac{a-d-c}{2c} = \frac{1-d}{c}.
\end{displaymath}
Then
\begin{displaymath}
h(z_0) = \frac{a z_0+b}{c z_0+d} = \frac{a \frac{1-d}{c}+b}{c \frac{1-d}{c}+d} =\frac{1+c-d}{c}=z_0+1
\end{displaymath}
and so with $f$ representing $X$ we have shown $f^{-1}h(z_0)=z_0$,  so $z_0$ is a fixed point, and further $X^{-1}W_{p/q}$ is parabolic as previously observed.  We have proved the following lemma.

\begin{lemma}\label{pfp}
  The point $\frac{1-d}{c}\in \partial \tilde{S}$,  a point of tangency of the isometric disks of $W_{p/q}$ and their unit translates, and with $\tilde{S}$ defined at Equation~(\ref{19}), is a parabolic fixed point.\qed
\end{lemma}

\subsection{Canonical peripheral quasidisks}

We want to analyse the pairing of the isometric circles of $W_{p/q}$ further.  We recall the notation from Equation~(\ref{isofarey}), so $ h_{p/q} $ maps $ \partial D_2 $
onto $ \partial D_1 $. In order to compute the distance between these discs, we want to find $z_0\in \partial D_2$ such that $|h(z_0)-z_0|$ is minimised (for convenience, we now drop the
subscripts $p/q$ as the slope is fixed). Let us make the following guess for the form of $ z_0 $:
\begin{displaymath}
  z_0=-\frac{d}{c}+ \frac{\zeta}{c}, \quad |\zeta|=1
\end{displaymath}
We calculate that
\begin{displaymath}
 h(z_0)= \frac{az_0+b}{cz_0+d} =  \frac{-\frac{ad}{c}+ \frac{a\zeta}{c}+b}{\zeta} = \frac{1}{c} \frac{a\zeta-1}{\zeta}  = \frac{a}{c} - \frac{\bar \zeta}{c}
\end{displaymath}
Then
\begin{displaymath}
 h(z_0)-z_0 =  \frac{a}{c} - \frac{\bar \zeta}{c}+\frac{d}{c} - \frac{\zeta}{c} = \frac{c+2}{c} - \frac{\zeta+\bar \zeta}{c} = 1+\frac{2-(\zeta+\bar \zeta)}{c}.
\end{displaymath}
Since $ c \neq 0 $, it suffices to minimise $|c||h(z_0)-z_0|$. Therefore we look at
\begin{displaymath}
  |c+2-(\zeta+\bar \zeta)| = | a+d -2\Re(\zeta)|.
\end{displaymath}

(Thus,  with $a+d=-2+it$, we find  $\zeta=-1$, and $h(z_0)=z_0$. That is, when $ \mu $ lies on the boundary of our conjectured neighbourhood (the inverse image of the line $ \Re z = -2 $), the isometric
circles become tangent.)

Now the line segment $ \ell_{p/q} $ joining $z_0$ to $h_{p/q}(z_0)$ will lie entirely in $\tilde{S}$ provided that the isometric disks have not twisted too far around.  In particular it is enough if the real part of the distance between the fixed points exceeds twice the radius of the isometric disks.  That is, if
\begin{displaymath}
  \Big|\Re\Big(\frac{a_{p/q}+d_{p/q}}{c_{p/q}}\Big)\Big| \geq \frac{2}{|c_{p/q}|}.
\end{displaymath}
Using Theorem~\ref{thmiso}, we calculate that
\begin{align*}
  \Re\Big(\frac{a_{p/q}+d_{p/q}}{c_{p/q}}\Big) &=  \Re\Big(\frac{c_{p/q}+2}{c_{p/q}}\Big) =1+ \Re\Big(\frac{2}{c_{p/q}}\Big) \\
                                                 &= 1+ \frac{2}{|c_{p/q}|^2} \Re\big(c_{p/q}\big)
\end{align*}
now requiring
\begin{displaymath}
  |c_{p/q}|^2+ 2 \Re\big(c_{p/q}\big) \geq 2|c_{p/q}|.
\end{displaymath}
This is true if $\Re(c_{p/q})\leq -4$.  Under these conditions the line segment
\begin{displaymath}
\ell_{p/q} = \big[-\frac{1+d_{p/q}}{c_{p/q}},\frac{a_{p/q}+1}{c_{p/q}}\big]
\end{displaymath}
now has the property that it lies entirely in $\tilde{S}$ with its endpoints on $\partial \tilde{S}$ and these endpoints are identified by $h_{p/q}$.

For convenience, introduce the notation $\Gamma_{p/q}=\langle f,h_{p/q} \rangle$ representing the group $\langle X,W_{p/q}\rangle\subset \PSL(2,\IC)$.  We identified a fundamental domain
$\tilde{S}$ for the action of $\Gamma_{p/q}$ on $\oC\setminus \Lambda_{p/q}$, where of course $\Lambda_{p/q}$ is the limit set of $\Gamma_{p/q}$.  The quotient
\begin{displaymath}
  (\oC\setminus \Lambda_{p/q})/\Gamma_{p/q}
\end{displaymath}
is the four-times punctured sphere $S^2_4$ (recall $\Gamma_{p/q}=\langle f,h_{p/q} \rangle = \langle f,f^{-1}h_{p/q} \rangle $ is a Schottky group generated by two parabolics).  The line segment $\ell_{p/q}$
projects to a simple closed curve in the homotopy class of $h_{p/q}$ representing a simple closed curve separating one pair of punctures from another. We remark that the projection of $\ell_{p/q}$ is smooth
away from one corner and the angle at that corner tends to $\pi$ as $t\to 0$. The Schottky lift of the projection of $\ell_{p/q}$ into $\IS^2_4$ is a quasiline through $\infty$ (we have no control on the
distortion here even though  we expect that we are a bounded hyperbolic distance from a Fuchsian group on the rational pleating ray $\mathcal{P}_{p/q}$,  so there is a nice quasiline which must pass through
the midpoint of $\ell_{p/q}$ for reasons of symmetry).  This quasiline must be
\begin{displaymath}
  L_{p/q} = \bigcup_{g\in \langle f,h_{p/q}\rangle} g(\ell_{p/q}).
\end{displaymath}
It consists of the translates of $\ell_{p/q}$ by $f^{n}$, $n\in \IZ$, together with images which lie in the union of the two isometric circles of $h_{p/q}$ and their integer translates.  We note that
\begin{displaymath}
  h_{p/q}(\infty)=\frac{a_{p/q}}{c_{p/q}}, \quad h_{p/q}^{-1}(\infty)=-\frac{d_{p/q}}{c_{p/q}}
\end{displaymath}
and these are parabolic fixed points on $L_{p/q}$ (conjugates of $f$) as well as being the centers of the isometric circles.  The parabolic fixed point we earlier identified at Lemma~\ref{pfp},
that is the point $z_\infty=\frac{1-d}{c}$,  also lies in $L_{p/q}$ and is not a conjugate of $f$ (since it is not conjugate in the abstract group $\langle x,y \rangle$ from which the rational
words come, again a purely topological consequence of the fact that they represent simple closed curves on the four-times punctured sphere.) The translates of $\ell_{p/q}$ under $\langle h_{p/q} \rangle$
form a log-spiral connecting the fixed points of $h_{p/q}$.  This is illustrated in the examples of Figure~\ref{fig:Spirals}.

\begin{figure}
  \centering
  \includegraphics[width=\textwidth]{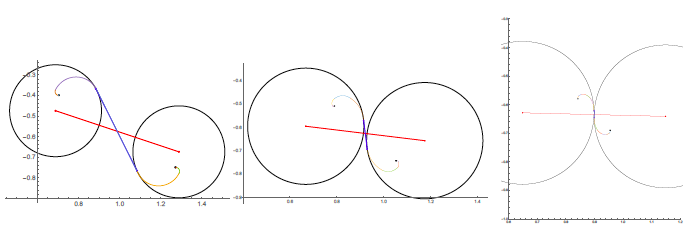}
  \caption{Red lines connect isometric circle centers,  spirals connect the fixed points of $h_{3/4}$.  With $\tr(W_{3.4})=-2+it$,  from left $t=2$,  $t=0.5$ and $t=0.1$.\label{fig:Spirals}}
\end{figure}

If we denote by $H_{p/q}^{\pm}$ the components of $\IC\setminus L_{p/q}$,  then
\begin{displaymath}
  H_{p/q}^{\pm}/\Gamma_{p/q}
\end{displaymath}
is a twice punctured disk with boundary a projection of $\ell_{p/q}$.

It is not relevant to the proof of the theorem, but we can give some bounds on the position of the invariant quasiline.
\begin{lemma}
  The invariant quasiline $L_{p/q}$ lies in the strip
  \begin{displaymath}
    \Big\{z\in \IC: \Im(\frac{a_{p/q}}{c_{p/q}})+\frac{1}{|c_{p/q}|}\leq \Im(z) \leq \Im(-\frac{d_{p/q}}{c_{p/q}}) - \frac{1}{|c_ {p/q}|}\Big\}
  \end{displaymath}
\end{lemma}
\begin{proof}
  By construction $\ell_{p/q}$ lies in,  and separates $\tilde{S}$. Its translates together with the translates of the isometric disks of $W_{p/q}$ separate both
  the ordinary set of $\langle f,h_{p/q} \rangle$ and  the plane into two parts.   The strip is the smallest horizontal strip containing the isometric circles of $W_{p/q}$.
  Note that $\Im(\frac{a_{p/q}}{c_{p/q}}) >  \Im(-\frac{d_{p/q}}{c_{p/q}})$ and that both are negative.  This particular fact holds if we choose, as we may,  $\mu$ to be in  the positive quadrant of $\IC$.
\end{proof}

Our computational investigations suggest that in fact the width of this strip can be improved to where the spiral ``turns over''.  This appears proportional to the difference of
the imaginary parts of the fixed points.  A  consequence would be that as $t\to 0$ the strip turns into a line and the quasilines $L_{p/q}$ converge to the line through the fixed
points of $h_{p/q}$,  which is a line in the limit set of the cusp group.

By analogy with Keen and Series we call the component $H_{p/q}$ of $\IC\setminus L_{p/q}$ which does not contain $0$ a \textit{canonical peripheral quasidisk} if
\begin{enumerate}
  \item $\Lambda(\Gamma_{p/q}) = \overline{H_{p/q}}\cap \Lambda(\Gamma_\mu)$, and
  \item $\tr(W_{p/q})\in \{z = x+iy \in \IC: x< -2 \}$.
\end{enumerate}
Notice that if we are at any value $\mu\in\mathcal{R}$,  then there \emph{some} slope $ p/q $ such that $ \Gamma_\mu $ admits the canonical peripheral quasidisk $ H_{p/q} $,
since each such group is quasiconformally conjugate to one on a pleating ray where there is such a peripheral circle ($F$-peripheral in \cite{KS}).  There seems to be no way
of guaranteeing that the large scale geometry of the boundary quasiline is bounded, but we do know that it is for $L_{p/q}$.

\subsection{Completing the proof}
We now give a series of lemmas imitating the proofs given for the case of a pleating ray in \cite{KS}.  Set  $S^*_{p/q}=\tilde{\cal S}\cap H_{p/q}$; this is a fundamental domain defined by the isometric
circles of $h_{p/q}$ and the line segment $\ell_{p/q}$.  Recall the parabolic cusp point given by Lemma~\ref{pfp}  in $\partial H_{p/q}$ (and also in $S^*_{p/q}$).  The following lemma is immediately
clear from  construction.

\begin{lemma}\label{priscan}
  An $F$-peripheral disk in the sense of \cite{KS} is a canonical peripheral quasidisk.
\end{lemma}
\begin{proof}
  In this case $h_{p/q}$ is hyperbolic, with disjoint isometric disks and $\ell_{p/q}$ is a segment of the line through its fixed points (and also through isometric circles) and orthogonal to them.
\end{proof}

The following lemma is analogous to the results of \cite{KS95}.
\begin{lemma}
  Fix a rational slope $p/q$. $L_{p/q}$ moves continuously with $\mu$ and the data $a_{p/q}, b_{p/q}, c_{p/q}$ and $d_{p/q}$, as does the associated fundamental domain $S^*_{p/q}$.
\end{lemma}
\begin{proof}
  In fact the defining points (vertices of $S^*_{p/q}$) move holomorphically,   but as a set $S^*_{p/q}$  does not.
\end{proof}

Next the analogue of \cite[Proposition 3.1]{KS}.
\begin{lemma}\label{lem:cpqopen}
  Fix a rational slope $p/q$. The set
  \begin{displaymath}
    \{\mu :\Gamma_\mu \mbox{ admits the canonical peripheral quasidisk $H_{p/q}$} \}
  \end{displaymath}
  is open.
\end{lemma}
\begin{proof}
  By definition $\tr(W_{p/q}) \in \{\Re(z)<-2\}$. Choose a small neighbourhood of $\mu$ so that this remains true.  That is $\tr(W_{p/q}(\mu')) \in \{\Re(z)<-2\}$ for $\mu'$ close to $\mu$. Each $\Gamma_\mu$ is geometrically finite,  and therefore each parabolic fixed point is doubly cusped \cite{AM,Mas}. Let $U$ be a horodisk neighbourhood of the parabolic fixed point in $\partial S^*_{p/q}$  (not $\infty$). As $\ell_{p,q}\in \partial H_{p/q}$ is in the domain of discontinuity for $\Gamma_{p,q}$ it is in the ordinary set of $\Gamma_\mu$ and projects to a loop bounding a doubly punctured disk in $S^2_4$.  It follows that $S^*_{p/q}\setminus U$ is compactly supported away from $\Lambda(\Gamma_\mu)$.  This limit set moves holomorphically and so for small time $t$ the varying $(S^*_{p/q})_t\setminus U_t$ lie in the ordinary set of $\Gamma_{\mu_t}$.  The images of $(S^*_{p/q})_t\setminus U_t$ under $(\Gamma_{p/q})_t$ tessellate $(H_{p/q})_t$ apart from the deleted cusp neighbourhoods which we now put back to find a canonical peripheral quasidisk $(H_{p/q})_t$.
\end{proof}

Next a version of \cite[Lemma 3.5]{KS}.

\begin{lemma}\label{lem6}
  Fix a rational slope $p/q$. Suppose that $\Gamma_\mu$ admits $H_{p/q}$,  a canonical peripheral quasidisk.  Then $\mu\in \mathcal{R}$.
\end{lemma}
\begin{proof}
  We have $\Gamma_{p/q}=\langle f,h_{p/q}\rangle=\langle f,f^{-1}h_{p/q}\rangle$.  As described earlier there is another group $\hat{\Gamma}_{p/q}$ generated by two
  parabolics in $\Gamma_\mu$ whose product is also $h_{p,q}$.  These groups are not conjugate in $\Gamma_\mu$ but are conjugate when the $\IZ_2$ symmetry that
  conjugates $X$ to $Y$ is added.  This symmetry leaves the limit set set-wise invariant.   Hence both are quasifuchsian with canonical peripheral quasidisks).  The
  remainder of the argument is as in \cite{KS}.  Briefly,   $H_{p/q}/\Gamma_{p/q} = H_{p/q}/\Gamma_{\mu}$ and, with the obvious notation,    $\hat{H}_{p/q}/\hat{\Gamma}_{p/q} = \hat{H}_{p/q}/\Gamma_{\mu}$
  are two different twice punctured disks in the quotient glued along a common boundary (a translation arc of $f$ which lies in ${H}_{p/q}\cap \hat{H}_{p/q}$). Then the quotient is $S^2_4$
  and $\mu\in\mathcal{R}$ by definition.
\end{proof}

It is really the next lemma where we use the fact that the quasidisks $H_{p/q}$ have bounded geometry.  Without this the invariant quasicircles for the peripheral disks could
either become space-filling curves or collapse entirely.  This indeed happens in general with the formation of $b$-groups,  or the geometrically infinite groups on the boundary of $\mathcal{R}$.

\begin{lemma}\label{lem:closed}
  Fix a rational slope $p/q$. Suppose that $\Gamma_{\mu_j}$ admits canonical peripheral quasidisks $H^j_{p/q}$,  and that $\tr(W_{p/q}^j)\to z_0$ with $\Re(z_0)<-2$.
  Then there is a subsequence $\mu_{j_k}\to \mu \in \mathcal{R}$ such that $\Gamma_\mu$ admits a canonical peripheral quasidisk.
\end{lemma}
\begin{proof}
  That $\tr(W_{p/q}^j)\to z_0$,  $\Re(z_0)<-2$ means that $a_{p/q},b_{p/q},c_{p/q}$ and $d_{p,q}$ all have finite limits,  that $c_{p/q}\not\to 0$ by Lemma~\ref{lemma3} and that
  therefore the invariant lines bounding $L_{p/q}$ also have a limiting height above and below.  It follows that there is an open set $U$ such that for $j$ sufficiently large $U\subset H_{p/q}^j$.
  Each of the groups $\Gamma_{\mu_j}$ is discrete (and free) and,  after passing to a subsequence if necessary,  the limit group $\Gamma_{\mu}$ is also discrete (and free).  Thus the
  ordinary set of $\Gamma_\mu$ must contain $U$.  By Lemma~\ref{lem6} we have $\mu_j\in \mathcal{R}$ and hence $\mu\in \overline{\mathcal{R}}$.  If $\mu\in \mathcal{R}$ we are done.
  Otherwise $\mu\in \partial\mathcal{R}$, and $\Gamma_\mu$ has nonempty ordinary set $\Omega_\mu=\IC\setminus \Lambda(\Gamma_\mu)$.  Since $\mu$ lies in the boundary of $\mathcal{R}$
  the quotient surface $\Omega/\Gamma_\mu$ can support no moduli.  $\Gamma_\mu$ is torsion free,  so the quotient is a union of triply punctured spheres and the point $\mu$ must be a
  cusp group (see \cite{MS} for these things). Notice that $h_{p/q}$ will have its fixed points in the boundary of a component of the ordinary set,  which are now round circles.
  Thus $\Gamma_{p/q}$ is Fuchsian,  $\tr(h_{p/q})$ is real and therefore $\tr(h_{p/q})\in (-\infty,-2)$.  But these groups lie on the pleating ray and in $\mathcal{R}$ and have
  $F$-peripheral disks.  This completes the proof.
\end{proof}

We now complete the proof of Theorem~\ref{main}. Fix a slope $ p/q $, and let $ \mathcal{N}_{p/q} $ be the set of $ \mu $ such that $ \Gamma_{\mu} $ admits the $ H_{p/q}$ canonical peripheral
quasidisk. By Lemma~\ref{lem6}, $ \mathcal{N}_{p/q} \subseteq \mathcal{R} $.

Consider the set $ \mathcal{Z}_{p/q} $ defined by
\begin{displaymath}
  \mathcal{Z}_{p/q} = \{ \mu \in \mathbb{C} : \Re P_{p/q}(\mu) < -2 \}.
\end{displaymath}
We make four observations.
\begin{enumerate}
  \item $ \mathcal{N}_{p/q} \subseteq \mathcal{Z}_{p/q} $ since, by definition, $ \Re \tr W_{p/q}(\mu) < -2 $ for $ \mu \in \mathcal{N}_{p/q} $;
  \item Note that $ \mathcal{N}_{p/q} $ is closed in $ \mathcal{Z}_{p/q} $ by Lemma~\ref{lem:closed}.
  \item By definition, $ \mathcal{Z}_{p/q} $ is open in $ \mathbb{C} $ (it is the inverse image of an open set); since $ \mathcal{N}_{p/q} $ is also open in $ \mathbb{C} $ (Lemma~\ref{lem:cpqopen})
        it is open in $ \mathcal{Z}_{p/q} $.
  \item Finally, $ \mathcal{N}_{p/q} \neq \emptyset $ by Lemma~\ref{priscan}.
\end{enumerate}
Thus $ \mathcal{N}_{p/q} $ is a union of non-empty connected components of $ \mathcal{Z}_{p/q} $ contained in $ \mathcal{R} $. By the Keen--Series theory, there are at most two such
connected components, namely the components corresponding to the pleating rays of asymptotic argument $ -\exp(\pi p/q) $ and $ -\exp(-\pi p/q) $ (Theorem~2.4 of \cite{SM}); and clearly
we hit both of these components. In any case, picking a branch of the inverse of $ P_{p/q} $ corresponding to these arguments will give a connected component of $ \mathcal{N}_{p/q} $,
and such a component is the desired neighbourhood of the cusp lying inside the Riley slice.

\printbibliography[heading=bibintoc]

@article{A1,
  author = {Hirotaka Akiyoshi and Ken'ichi Ohshika and John Parker and Makoto Sakuma and Han Yoshida},
  title = {Classification of non-free Kleinian groups generated by two parabolic transformations},
  journal = {Transactions of the American Mathematical Society},
  volume = {374},
  year = {2021},
  pages = {1765-1814},
  url = {https://arxiv.org/abs/2001.09564}
}

@article{A2,
  author = {Shunsuke Aimi and Donghi Lee and Shunsuke Sakai and Makoto Sakuma},
  title = {Classification of parabolic generating pairs of Kleinian groups with two parabolic generators},
  journal = {Rendiconti dell'Istituto di Matematica dell'Università di Trieste},
  volume = {52},
  year = {2020},
  pages = {477-511},
  url = {https://arxiv.org/abs/2001.11662}
}

@book{A3,
  author = {Hirotaka Akiyoshi and Makoto Sakuma and Masaaki Wada and Yasushi Yamashita},
  title = {Punctured torus groups and 2-bridge knot groups I},
  publisher = {Springer},
  series = {Lecture Notes in Mathematics},
  number = {1909},
  year = {2007}
}

@article{Bers,
  author = {Lipman Bers},
  title = {On boundaries of Teichmüller spaces and on Kleinian groups I},
  journal = {Annals of Mathematics},
  volume = {2},
  issue = {91},
  pages = {570-600},
  year = {1970}
}

@article{Brenner,
  author = {J. L. Brenner},
  title = {Quelques groupes libres de matrices},
  journal = {Comptes Rendus de l'Académie des Sciences},
  volume = {241},
  pages = {1689--1691},
  year = {1955}
}

@article{CJR,
  author = {Bomshik Chang and S. A. Jennings and Rimhak Ree},
  title = {On certain pairs of matrices which generate free groups},
  journal = {Canadian Journal of Mathematics},
  volume = {10},
  year = {1958},
  pages = {279-284}
}

@article{CMM,
  author = {M. D. E. Conder and C. Maclachlan and G. J. Martin and E. A. O'Brien},
  title = {2-generator arithmetic Kleinian groups III},
  journal = {Mathematica Scandinavica},
  volume = {90},
  issue = {2},
  year = {2002},
  pages = {161-179}
}

@misc{Farey,
  author = {Cmglee (Wikimedia user)},
  url = {https://commons.wikimedia.org/w/index.php?curid=59832325},
  note = {CC BY-SA 4.0},
}

@article{FR,
  author = {Valérie Flammang and Georges Rhin},
  title = {Algebraic integers whose conjugates all lie in an ellipse},
  journal = {Mathematics of Computation},
  volume = {74},
  issue = {252},
  pages = {2007-2015},
  year = {2005}
}

@article{EKK,
  author = {C. J. Earle and I. Kra and S. L. Krushkal},
  title = {Holomorphic Motions and Teichmüller Spaces},
  journal = {Transactions of the American Mathematical Society},
  volume = {343},
  issue = {2},
  pages = {927-948},
  year = {1994}
}

@article{GGM,
  author={F. W. Gehring and J. Gilman and G. J. Martin},
  title={Kleinian groups with real parameters},
  journal={Communications in Contemporary Mathematics},
  volume = {3},
  year = {2001},
  pages = {163-186}
}

@article{GMM,
  author = {F. W. Gehring and C. Maclachlan and G. J. Martin},
  title = {Two-generator arithmetic Kleinian groups II},
  journal = {Bulletin of the London Mathematical Society},
  volume = {30},
  issue = {3},
  year = {1998},
  pages = {258-266}
}

@article{GMMR,
  author = {F. W. Gehring and C. Maclachlan and G. J. Martin and A. W. Reid},
  title = {Arithmeticity, Discreteness and Volume},
  journal = {Transactions of the American Mathematical Society},
  volume = {347},
  year = {1997},
  pages = {3611-3643}
}

@article{GM,
  author = {F. W. Gehring and G. J. Martin},
  title = {Commutators, collars and the geometry of M\"obius groups},
  journal = {Journal d'Analyse Math\'ematique},
  volume = {63},
  year = {1994},
  pages = {175-219}
}

@article {Jorg,
  author = {Troels J\o rgensen},
  title = {On discrete groups of M\"{o}bius transformations},
  journal = {American Journal of Mathematics},
  volume = {98},
  year = {1976},
  pages = {739-749}
}

@article{KS,
  author = {Linda Keen and Caroline Series},
  title = {The Riley slice of Schottky space},
  journal = {Proceedings of the London Mathematics Society},
  volume = {3},
  issue = {69},
  number = {1},
  pages = {72-90},
  year = {1994}
}

@article{KS95,
  author = {Linda Keen and Caroline Series},
  title = {Continuity of convex hull boundaries},
  journal = {Pacific Journal of Mathematics},
  volume = {168},
  number = {1},
  pages = {183-206},
  year = {1995}
}

@inbook{SM,
  author = {Yohei Komori and Caroline Series},
  title = {The Riley slice revisited},
  booktitle = {The Epstein birthday schrift},
  editors = {Igor Rivin and Colin Rourke and Caroline Series},
  series = {Geometry and Topology Monographs},
  volume = {1},
  pages = {303-316},
  year = {1998}
}

@book{silverman,
  author = {Joseph H. Silverman},
  title = {The arithmetic of dynamical systems},
  publisher = {Springer},
  series = {Graduate Texts in Mathematics},
  number = {241},
  year = {2007}
}

@article{Leut,
  author = {A. Leutbecher},
  title = {\"Uber die Heckeschen Gruppen $G(\lambda)$},
  journal = {Abhandlungen aus dem Mathematischen Seminar der Universität Hamburg},
  volume = {81},
  year = {1967},
  pages = {199-205}
}

@article{LU,
  author = {R. C. Lyndon and J. L. Ullman},
  title = {Groups generated by two parabolic linear fractional transformations},
  journal = {Canadian Journal of Mathematics},
  volume = {21},
  year = {1969},
  pages = {1388-1403}
}

@article{Lyb,
  author = {M. Yu. Lyubich and V. V. Suvorov},
  title = {Free subgroups of $SL(2, \IC)$ with two parabolic generators},
  journal = {Journal of Soviet Mathematics},
  volume = {141},
  year = {1988},
  pages = {976-979}
}

@article{AM,
  author = {Albert Marden},
  title = {The geometry of finitely generated Kleinian groups},
  journal = {Annals of Mathematics},
  volume = {99},
  year = {1974},
  pages = {383-462}
}

@article{M1,
  author = {G. J. Martin},
  title={Nondiscrete parabolic characters of the free group F2: supergroup density and Nielsen classes in the complement of the Riley slice},
  journal = {Journal of the London Mathematical Society},
  volume = {103},
  year = {2021},
  pages = {1402-1414}
}

@article{MM,
  author = {C. Maclachlan and G.J. Martin},
  title ={The $(6,p)$-arithmetic hyperbolic lattices in dimension $3$},
  journal = {Pure and Applied Mathematics Quarterly},
  volume = {7},
  issue = {2},
  note = {Special Issue: In honor of Frederick W. Gehring},
  pages = {365-382}
}

@misc{MSY,
  author = {G.J. Martin and K. Selahi and Y. Tamashita},
  title ={The $(4,p)$-arithmetic hyperbolic lattices in dimension $3$},
  note = {To appear}
}

@book{Mas,
  author = {Bernard Maskit},
  title = {Kleinian groups},
  publisher = {Springer-Verlag},
  year = {1987},
  series = {Grundlehren der mathematischen Wissenshaften},
  vulume = {287}
}

@article{MS,
  author = {Bernard Maskit and Gadde Swarup},
  title = {Two parabolic generator Kleinian groups},
  journal = {Israel Journal of Mathematics},
  volume = {64},
  number = {3},
  pages = {257-266},
  year = {1989}
}

@article{McMCusps,
  author = {Curt McMullen},
  title = {Cusps are dense},
  journal = {Annals of Mathematics},
  number = {133},
  pages = {217-247},
  year = {1991}
}

@article{Sanov,
  author = {L. N. Sanov},
  title = {A property of a representation of a free group},
  journal = {Doklady Akademii Nauk SSSR},
  volume = {57},
  year = {1947},
  pages = {657-259}
}

@article{Slod1,
  author = {Zbigniew Slodkowski},
  title = {Holomorphic motions and polynomial hulls},
  journal = {Proceedings of the American Mathematical Society},
  volume = {111},
  year = {1991},
  pages = {347-355}
}

@article{Slod2,
  author = {Zbigniew Slodkowski},
  title = {Natural extensions of holomorphic motions},
  journal = {Journal of Geometric Analysis},
  volume = {7},
  issue = {4},
  year = {1997},
  pages = {637-651}
}

@misc{EMSElliptic,
  author = {A. Elzenaar and G.J. Martin and J. Schillewaert},
  title ={Neighbourhoods of cusps in the elliptic Riley slices},
  note = {In preparation}
}

@misc{EMSCombinatorics,
  author = {A. Elzenaar and G.J. Martin and J. Schillewaert},
  title ={The combinatorics of Farey words and their traces},
  note = {In preparation}
}

\end{document}